\documentclass[a4paper,11 pt,reqno]{amsart}
\usepackage[latin1]{inputenc}
\usepackage{amsmath,amsthm,amssymb,amsfonts,amscd,amstext}
\usepackage{latexsym}
\usepackage{color}
\usepackage{xcolor}
\usepackage{graphicx}
\usepackage{mathrsfs}
\usepackage{tikz}
\usepackage{comment}
\usepackage{enumitem}
\usepackage[left=2.3cm,right=2.3cm,top=2.5cm,bottom=2.5cm]{geometry}
\makeatletter
\newcommand{\myitem}[1]{%
	\item[#1]\protected@edef\@currentlabel{#1}%
}
\makeatother
\usepackage[hyperpageref]{backref}
\usepackage[
colorlinks=true,
linkcolor=blue,
citecolor=blue,
urlcolor=green,
backref=page
]{hyperref}

\usepackage[
colorlinks=true,
linkcolor=blue,
citecolor=blue,
urlcolor=green,
backref=page
]{hyperref}
\renewcommand*{\backref}[1]{}
\renewcommand*{\backrefalt}[4]{%
	\ifcase #1 %
	\or #2%
	\else #2%
	\fi
}

\allowdisplaybreaks
\newtheorem{theorem}{Theorem}[section]
\newtheorem{remark}{Remark}
\newtheorem{lemma}[theorem]{Lemma}
\newtheorem{proposition}[theorem]{Proposition}
\newtheorem{corollary}[theorem]{Corollary}
\newtheorem{definition}[theorem]{Definition}

\numberwithin{equation}{section}

\newcommand{\eps}{\varepsilon}

\newcommand{\R}{\mathbb{R}}

\newenvironment{myproof}[2] {\paragraph{Proof of {#1} {#2} :}}{\hfill$\square$}

\title[Normalized solutions for a mixed Choquard equation]
{Normalized Solutions for a Mixed Local-Nonlocal Choquard Equation with Exponential Growth in \(\R^2\): A Moser-Type Approach }

\author[Nidhi Nidhi]{Nidhi Nidhi}
\address[N.~Nidhi]{Department of Mathematics, Indian Institute of Technology Delhi, 110016, India}
\email{nidhi.kaushik2809@gmail.com}

\author[L.~Sharma]{L.~Sharma}
\address[L.~Sharma]{Department of Mathematics, Indian Institute of Technology Delhi, 110016, India}
\email{slovelesh.96@gmail.com}

\author[K.~Sreenadh]{K.~Sreenadh$^{*}$}
\address[K.~Sreenadh]{Department of Mathematics, Indian Institute of Technology Delhi, 110016, India}
\email{sreenadh@maths.iitd.ac.in}

\thanks{$^{*}$Corresponding author.}


\begin{document}
	
	\begin{abstract}
		\noindent
		In this paper, we study the existence of normalized solutions to the
		mixed nonlinear Choquard equation with exponential growth
		\[
		\left\{
		\begin{aligned}
			\mathcal{L}u+\lambda u
			&=
			\Lambda (I_{\alpha}*F(u))f(u)
			&&\text{in }\mathbb{R}^{2},\\
			\int_{\mathbb{R}^{2}}|u|^{2}\,dx
			&=
			a^{2},
		\end{aligned}
		\right.
		\]
		where
		\(\mathcal{L}=-\Delta+(-\Delta)^s\), \(0<s<1\),
		\(a>0\), \(I_{\alpha}\) is the Riesz potential of order
		\(\alpha\in(0,2)\), \(\Lambda>0\), and
		\(\lambda\in\mathbb{R}\) arises as a Lagrange multiplier.
		The nonlinearity has critical exponential growth in the sense of the
		Trudinger-Moser inequality. Using variational methods, the associated
		Poho\v{z}aev manifold, and suitable Moser-type estimates,  we establish {the existence of a normalized solution of mountain-pass type with associated non zero Lagrange multiplier.
		} Under an additional monotonicity assumption \ref{f6} on the nonlinearity, we
		further show that the obtained solution is a normalized ground state.
		We also establish the regularity properties required for the proof
		of the corresponding Poho\v{z}aev identity.
		
	\end{abstract}
	
	\maketitle
	
	\section{Introduction}
	\noindent The aim of this paper is to analyze a mixed nonlinear Choquard equation with exponential growth, focusing on the existence of normalized solutions
	\begin{equation}\label{aa}
		\left\{
		\begin{aligned}
			\mathcal{L}u+\lambda u \; &=\; \Lambda(I_{\alpha}\ast F(u))F'(u), \text{ in }\mathbb{R}^{2},\\
			\int_{\mathbb{R}^{2}}|u|^{2}\,dx \; &=\; a^{2},
		\end{aligned}
		\right.\tag{$P_{\lambda}$}
	\end{equation}
	where $a>0$ is a constant, $\Lambda>0$ is a parameter, $I_{\alpha}$ is the Riesz potential of order $\alpha\in (0,2)$ defined as:
	\begin{eqnarray*}
		\begin{aligned}
			I_{\alpha}(x)=\frac{\Gamma(\frac{2-\alpha}{2})}{\Gamma(\frac{\alpha}{2})2^{\alpha}\pi|x|^{2-\alpha}}:=\frac{A_{\alpha}}{|x|^{2-\alpha}}, \quad x\in \mathbb{R}^{2}\backslash\{0\},
		\end{aligned}
	\end{eqnarray*}
	the mixed diffusion type operator $\mathcal{L}=-\Delta+(-\Delta)^s$ with
	$$ (-\Delta)^su(x)=\frac{C(2,s)}{2}\text{P.V.}\int_{\mathbb{R}^2}\frac{u(x)-u(y)}{|x-y|^{2+2s}}dy \text{ for } s\in (0,1),$$
	here $C(2, s)$ is the normalizing constant given by 
	$$C(2,s)=2\left(\int_{\mathbb{R}^2}\frac{1-\cos(x)}{|x|^{2+2s}}dx\right)^{-1},$$
	and P.V. is the abbreviation for principal value. For the sake of convenience, we will consider \( C(2,s) = 2 \).
	Denoting
	\[
	F(t):=\int_0^t f(\tau)\,d\tau,
	\]
	we assume that $f$ satisfies the following
	conditions:
	
	\begin{enumerate}
		
		\myitem{\textnormal{(f\(_0\))}}\label{f0}
		\( f\in C(\mathbb R,\mathbb R_+), ~~
		f(t)=0
		\quad \forall~t\leq 0;
		\)
		
		\myitem{\textnormal{(f\(_1\))}}\label{f1}
		\(
		f(t)
		=
		o\left(t^{1+\frac{\alpha}{2}}\right)
		\text{ as }t\to0^+;
		\)
		
		\myitem{\textnormal{(f\(_2\))}}\label{f2}
		there exists \(\gamma_0>0\) such that
		\[
		\lim_{t\to+\infty}
		\frac{f(t)}{e^{\gamma t^2}}
		=
		\begin{cases}
			0,
			& \text{if }\gamma>\gamma_0,\\[1mm]
			+\infty,
			& \text{if }\gamma<\gamma_0;
		\end{cases}
		\]
		
		\myitem{\textnormal{(f\(_3\))}}\label{f3}
		there exists
		\(
		\mu>2+\frac{\alpha}{2}
		\)
		such that
		\[
		0<\mu F(t)\leq f(t)t
		\quad \forall~t>0.
		\]
		
		\myitem{\textnormal{(f\(_4\))}}\label{f4}
		\(\displaystyle
		\lim_{t\to+\infty}
		\frac{F(t)}{t f(t)}
		=
		0.
		\)
		\myitem{\textnormal{(f\(_5\))}}\label{f5}
		there exists \(\beta_0>0\) such that
		\[
		\liminf_{t\to+\infty}
		\frac{t f(t)}{e^{\gamma_0t^2}}
		\geq\beta_0;
		\]
		
		\myitem{\textnormal{(f\(_6\))}}\label{f6}
		the function
		\[
		\widetilde F(t)
		:=
		f(t)t-\left(\frac{2+\alpha}{2}\right)F(t)
		\]
		is nondecreasing on \((0,+\infty)\).
		
	\end{enumerate}
	
	\noindent The above conditions \ref{f0}-\ref{f5} will be assumed
	throughout the paper. Condition \ref{f6} be imposed only
	when proving the normalized ground-state solution in Theorem \ref{T2}.
	\subsection*{Consequences of the assumptions}
	We point out some consequences of
	\ref{f0}-\ref{f6} which will be
	used throughout the paper in our analysis.
	
	\smallskip
	
	\noindent\textnormal{(i)}
	By \ref{f1} and \ref{f2}, for every \(\varepsilon>0\), \(q>1\), and
	\(\gamma>\gamma_0\), there exists a constant
	\(
	\kappa_\varepsilon
	=
	\kappa_\varepsilon(q,\gamma,\varepsilon)>0
	\)
	such that
	\begin{equation}\label{nonb}
		|f(t)|
		\leq
		\varepsilon |t|^{1+\frac{\alpha}{2}}
		+
		\kappa_\varepsilon |t|^{q-1}
		\left(e^{\gamma |t|^2}-1\right),
		\qquad \forall~
		t\in\mathbb{R},
	\end{equation}
	and
	\begin{equation}\label{nona}
		|F(t)|
		\leq
		\varepsilon |t|^{2+\frac{\alpha}{2}}
		+
		\kappa_\varepsilon |t|^q
		\left(e^{\gamma |t|^2}-1\right),
		\qquad \forall~
		t\in\mathbb{R}.
	\end{equation}
	
	\smallskip
	
	\noindent\textnormal{(ii)}
	Assumption \ref{f4} is a condition at infinity that will be used in
	the gain compactness analysis. In particular, it is weaker than the
	classical condition requiring the existence of constants
	\(t_0,L_0>0\) such that:
	\begin{flushleft}
		\(\textnormal{\textbf{(f\(_4'\))}}\)\quad
		\(F(t)\leq L_0f(t)\) \quad
		for every \(t\geq t_0\).
	\end{flushleft}
	Moreover, by \ref{f1} and \ref{f4}, for every \(\varepsilon>0\), there
	exists \(C_\varepsilon>0\) such that
	\begin{equation}\label{F-f-estimate}
		F(t)
		\leq
		\varepsilon f(t)t
		+
		C_\varepsilon |t|^{2+\frac{\alpha}{2}},
		\qquad~~\forall~
		t\in\mathbb{R}.
	\end{equation}
	
	\smallskip
	
	\noindent\textnormal{(iii)}
	Condition \ref{f5} is the
	de Figueiredo-Miyagaki-Ruf condition, see \cite{DMR} . It is important in the
	Moser-sequence argument used to estimate the mountain-pass level.
	In particular, this assumption replaces the stronger global lower bound with \(\mu> 2+ \frac{\alpha}{2}\)
	\begin{equation}\label{glb}
		F(t)\geq c|t|^\mu
		\text{ for all }t>0.
	\end{equation}
	\noindent\textnormal{(iv)}
	As a consequence of \ref{f5} and using l'H\^opital's rule, one has
	\begin{equation}\label{lower-F-critical}
		\lim_{t\to+\infty}
		\frac{t^2F(t)}{e^{\gamma_0t^2}}
		\geq
		\frac{\beta_0}{2\gamma_0}.
	\end{equation}

	\noindent\textnormal{(v)}
	Condition\ref{f6} is  used in the
	ground-state analysis. More precisely, it ensures the monotonicity
	property required to study the fibre map associated with the
	\(L^2\)-preserving scaling and to identify the mountain-pass solution
	as a normalized ground state.
	
	\begin{remark}
		We note that functions satisfying conditions \ref{f0}-\ref{f6} are available in the literature; see \cite{HuangRomani2025}.
	\end{remark}
	

	{\noindent Equations involving nonlinearity of the form $(I_{\alpha}*F(u))F'(u)$ are called {\it Choquard equation}, as in 1976, Choquard, at the Symposium on Coulomb Systems  
		utilised the energy functional associated to equation 
		\begin{equation}
			\left\{ \begin{array}{rl}   	
				& 	-\Delta u  +  u = (I_2*|u|^2)u\;\;\text{in } \mathbb{R}^3,\\
				&  	u \in H^1(\mathbb{R}^3),
			\end{array}
			\right.
		\end{equation}
		to examine a viable approximation to Hartree-Fock theory for a one-component plasma (see \cite{lieb1977existence}). The equation has various other applications in quantum physics, for instance, it is used to characterise an electron confined within its own vacancy, see \cite{penrose1996gravity} and related sources. Several works have ever since conducted research on the existence, multiplicity, and qualitative characteristics of the solution to the Choquard-type equations
		as detailed in \cite{filippucci2020singular, Moroz2013groundstates, liu2022another}. 
		
		\noindent The study of mixed operators of the type $\mathcal{L}$ as in the problem \eqref{aa} is motivated by several applications where such a kind of operators are naturally generated, including the theory of optimal searching, biomathematics, and animal foraging, for which we refer to \cite{dipierro2022non, dipierro2021description}. In applied sciences, they are used for investigating the changes in physical phenomena that have both local and nonlocal effects. For instance, they are present in bi-modal power law distribution systems, see \cite{MR4225516}. Furthermore, they are present in models that are derived from the combination of two distinct scaled stochastic processes.
		\noindent The aim of our work is to find the standing wave solution to Nonlinear Schr$\ddot{\text{o}}$dinger Equation. A standing wave solution for a nonlinear Schr$\ddot{\text{o}}$dinger (NLS) equation driven by mixed local and nonlocal operators is given as follows:
		\begin{equation}\label{NLS_eq}
			i\frac{\partial \psi}{\partial t}=-\Delta \psi +(-\Delta)^s\psi -\mu (I_{\alpha}*F(\psi))f(\psi),
		\end{equation}
		is of the form $\psi(x,t)=e^{-i\lambda t}u(x)$, where $\lambda\in \mathbb{R}$ and $u\in H^1(\mathbb{R}^N)$ solves:
		\begin{equation}\label{1.2}
			\mathcal{L}u=\lambda u +\mu (I_{\alpha}*F(u))F'(u)  \text{ in }\mathbb{R}^N.
		\end{equation}
		While addressing solutions to \eqref{1.2}, there exist two schools of thought. The initial approach involves fixing a $\lambda \in \mathbb{R}$ and thereafter looking for the critical points of the associated functional; this method has already been extensively employed, see for instance \cite{biagi2025brezis, Anthal2026Ground, Moroz2013groundstates} and references therein. The other method is to fix the $L^2$-norm and look for the solution to the following constrained problem:
		\begin{equation*}
			\left\{
			\begin{array}{rcl}
				-\Delta	u+(-\Delta)^su & = & \lambda u +\mu(I_{\alpha}*F(u))F'(u) \text{ in } \mathbb{R}^N,\\
				\| u \|_2^2 & = & \tau^2,
			\end{array}
			\right.
		\end{equation*}
		called the {\it normalized solution} or the solution with prescribed mass. The physical motivation to study this problem arises from the fact that its solution provides stationary states of a nonlinear Schr$\ddot{\text{o}}$dinger equation with a prescribed $L^2$-norm, which represents an entity that satisfies dynamic conditions and has a fixed mass. The study of normalized solutions can be dated back to the pioneering work in \cite{Jeanjean1997}, where Jeanjean obtained the existence of radial solutions for
		\begin{equation}\label{Norm_sol}
			\left\{ \begin{array}{rl}   	
				-\Delta u & =  \lambda u +g(u)\;\;\text{in } \mathbb{R}^N,\\
				\| u \|_2^2 & =  c,
			\end{array}
			\right.
		\end{equation} 
		under some assumptions on $g$. In \cite{bartsch2012normalized}, the existence of infinite solutions to \eqref{Norm_sol} under same assumptions has been shown.
		Further in \cite{noris2015existence}, the normalized solutions are discussed and described in the case of bounded domains with Dirichlet boundary conditions. Considering the domain to be the unit ball and $g(x)=|x|^{p-1}x$, the existence of normalized solutions has been seen for $p$ lying in $(1,1+\frac{4}{N}),\; (1+\frac{4}{N}, 2^*-1)$ and $p=1+\frac{4}{N}$ under some conditions on $c$. 
		Moreover, the problem in general bounded domains has been dealt by the authors in \cite{pierotti2017normalized}. Recently such problems involving fixed mass constraints have attracted many researchers; see, for instance, the work of \cite{gou2018multiple,bartsch2016normalized,bartsch2018normalized, noris2019normalized,Giacomoni2025Normalized, Nidhi2025Normalized_Asymp}, where authors studied the existence, multiplicity and regularity of normalized solutions for some nonlinear Schr$\ddot{\text{o}}$dinger equations with several local and nonlocal nonlinearities. Normalized solutions are also taken into consideration in the study of quadratic ergodic mean field games systems, see in particular \cite{pellacci2021normalized}.\\
		\noindent Almost all of the above-discussed studies deal in higher dimensions, that is, for $N>2$. For $N=2$, the critical exponent turns out to be $\infty$, which complicates the problem further and, at the same time, makes it more intriguing for researchers to work on, as it suggests that traditional methods of analysis may break down and require new approaches to comprehend the behaviour of the system. {In dimension two, the role of the Sobolev critical exponent is
			replaced by the critical exponential growth associated with the
			Trudinger-Moser inequality. A pioneering variational study of elliptic
			equations with critical exponential nonlinearities was carried out by
			Adimurthi \cite{Adimurthi1989,Adimurthi1990}. We also cite the work of
			Adimurthi and Yadava \cite{AdimurthiYadava1990} for multiplicity results
			in the two-dimensional critical case.}
		Later, de Figueiredo, Miyagaki and Ruf \cite{DMR} developed a
		variational framework for planar elliptic equations 
		where they studied the problem:
		\begin{equation*}
			\left\{
			\begin{array}{rl}
				-\Delta u & = f(x,u) \text{ in } \Omega \subset\mathbb{R}^2;\\
				u &= 0 \text{ on } \partial \Omega, 
			\end{array}
			\right.
		\end{equation*}
		and proved the existence of solution for critical growth of $f$; and multiplicity of solutions when $f$ satisfies subcritical growth conditions. Further, Shen and Squassina in \cite{Shen2025Concentrating} discussed the existence and concentrating behaviour of the normalized solution to a quite general problem. Moreover, the normalized solutions for some Kirchhoff-Choquard equations have also been studied; for instance, see the works in \cite{Yi2026Normalized}, \cite{Jin2026Solutions} and references therein. Motivated by the above well-established literature, we would like to study the existence and regularity of normalized solutions to an exponential growth. The Choquard equation driven by a mixed local-nonlocal operator in $\mathbb{R}^2$.

		\noindent
		Considering the solution space \(H^1(\mathbb{R}^2)\) and the
		variational framework introduced in Section~\ref{preliminaries}, we begin with highlighting our main results:

		\begin{theorem}\label{T1}
			Assume that \(f\) satisfies
			\ref{f0}-\ref{f5}.
			Then, for every \(a>0\), there exist
			\(u_a\in H^1_{\rm rad}(\mathbb R^2)\setminus\{0\}\) and
			\(0\neq \lambda_a\in\mathbb R\) such that $\left(u_a,\lambda_a\right)$ solves \eqref{aa}.
		Moreover,
		\[
		\lambda_a
		<
		\frac{(2+\alpha)^2\pi}
		{\gamma_0a^2(2\mu-4-\alpha)}.
		\]
		
	\end{theorem}
	
	\begin{remark} The upper bound for \(\lambda_a\) depends on the parameters \(\gamma_0\), \(a\), \(\alpha\), and \(\mu>2+\frac{\alpha}{2}\). 
	\end{remark}
	
	\noindent Under the additional assumption \ref{f6}, we show that
	the normalized solution obtained in Theorem \ref{T1} is a normalized
	ground state solution.

	\begin{theorem}\label{T2}
		Assume that \(f\) satisfies
		\ref{f0}-\ref{f6}.
		Then the solution \(u_a\) obtained in Theorem~\ref{T1} is a normalized
		ground state solution of \eqref{aa}.
	\end{theorem}

	\noindent	\textbf{Main Novelty and Strategy of the Proofs:}
	Before going into the proofs, we briefly explain the main idea of our approach. 
	The main novelty of the present work lies in the study of normalized solutions 
	for a mixed local-nonlocal Choquard equation with exponential growth in $\mathbb{R}^2$. To the best of our knowledge, this is the first result combining 
	a mixed operator of the form $ \mathcal{L}= -\Delta + (-\Delta)^s$, a nonlocal Choquard-type nonlinearity, and exponential growth in the sense of the Trudinger-Moser 
	inequality under a mass constraint. This setting introduces significant analytical difficulties due to the simultaneous presence of exponential nonlinearity, nonlocal convolution terms, and the lack of compactness in $\mathbb{R}^2$.
	
	\noindent The problem has a variational structure, so normalized solutions of \eqref{aa} can be obtained as critical points of the functional $J$ defined in \eqref{J}, 
	restricted to the $L^2$-sphere $\mathcal{S}(a)$. Since the equation is autonomous, we work in the radial space $H^1_{rad}(\mathbb{R}^2)$, where the embedding into $L^q(\mathbb{R}^2)$ is compact for all $q\in (2,\infty)$. By Palais' principle of symmetric criticality (see \cite{Palais1979}), solutions found in the radial space are also solutions in the whole space $H^1(\mathbb{R}^2)$. Under assumptions \ref{f1} - \ref{f2}, the functional $J$ is well-defined and of class $C^1$, thanks to Proposition \ref{PRa} and the Hardy-Littlewood-Sobolev inequality (see Proposition \ref{mixedTM}). Moreover, the nature of the problem gives rise to a mountain pass geometry (Lemma 
	\ref{P1_mixed}). Following \cite{Jeanjean1997}, this ensures the existence of a Palais-Smale Poho\v{z}aev $(PSP)$ sequence at the level $\gamma(a)$, namely a Palais-Smale sequence \(u_n\) at level \(\gamma(a)\) such that \({P}(u_n) \to 0\) as \(n\to +\infty\), where \({P}\) is the Poho\v{z}aev functional associated with \eqref{aa}, see \eqref{Pohozaev}.
	This property is important in establishing the boundedness of the sequence
	\(\{u_n\}\).

	\noindent A key novelty in our approach is the adaptation of the Poho\v{z}aev manifold technique to the mixed local-nonlocal framework with exponential growth, which requires delicate estimates compatible with the Trudinger-Moser setting.  
	
	A further technical difficulty arises in the Moser-sequence analysis due
	to the presence of the additional fractional energy term. To overcome
	this difficulty, we construct in Lemma~\ref{estimates_w_m} a mass-corrected Moser
	sequence of the form
	\(
	w_m=M_m+c_m\eta_m,
	\)
	where \(M_m\) is the classical Moser function and \(\eta_m\) is a
	suitably rescaled radial function whose support is separated from that
	of \(M_m\). The correction term carries the missing \(L^2\)-mass while
	remaining asymptotically negligible in the energy. More precisely, the
	construction yields
	\[
	\|\nabla w_m\|_2^2
	=
	1+o\!\left(\frac1{\ln m}\right),
	\qquad
	[w_m]_s^2
	=
	o\!\left(\frac1{(\ln m)^\delta}\right)
	\]
	for some \(\delta\in(0,1-s)\). The separation of the supports also makes
	the interaction term in the fractional seminorm nonpositive. This
	construction allows us to adapt the Moser-level estimate to the mixed
	local-nonlocal setting and plays a key role in estimating the
	mountain-pass level in Lemma~\ref{MP-level} by using assumption
	\ref{f5} and Moser sequences. {
		It is essential in the compactness analysis of the
		\((PSP)\)-sequence and in obtaining a nontrivial weak limit in Section \ref{PSa}.
		When the associated Lagrange multiplier is nonzero, the prescribed
		\(L^2\)-mass is recovered and the convergence becomes strong in
		\(H^1(\mathbb R^2)\).
	}
	\noindent Finally, under the additional monotonicity assumption
	\ref{f6}, we prove that the functional along each
	relevant fibre path attains its maximum at a unique point belonging to
	the Poho\v{z}aev manifold
	\(
	\left\{
	u\in\mathcal S(a): P(u)=0
	\right\}.
	\)
	Once this structure is established, we compare the mountain pass level $\gamma(a)$ with the least energy level $m(a)$, showing that they coincide and thus characterize the ground state energy.
	
	\vspace{0.2cm}
	
	\noindent\textbf{Organization of the paper.} In section~\ref{preliminaries}, we introduce the
	variational framework, establish the regularity properties of weak
	solutions, and derive the Pohozaev identity. In Section~\ref{sec3.1}, we establish
	the mountain pass geometry and introduce the corresponding minimax
	level. Section~\ref{PSa} is devoted to the construction and analysis of the
	Palais-Smale-Pohozaev sequence, including the Moser-type estimate for
	the mountain-pass level and the required compactness results. Finally,
	in Section~\ref{proofs}, we prove Theorems~\ref{T1} and~\ref{T2}.
	
	
	\vspace{0.5cm}
	
	\noindent\textbf{Notation.}
	Throughout the paper, unless otherwise stated, we use the following
	notation:
	\begin{enumerate}[label=\textnormal{(\roman*)}]
		
		\item \(\|\cdot\|\) denotes the norm of \(H^1(\mathbb{R}^2)\);
		
		\item \(\|\cdot\|_p\) denotes the usual norm in
		\(L^p(\mathbb{R}^2)\), for \(p\in[1,+\infty]\);
		\item For any $u\in H^1(\mathbb{R}^2)$, we define:
		$$T(u):=\| \nabla u \|_2^2+[u]_s^2;$$
		\item for \(a\geq1\), the expression \(|u|^a\) denotes the pointwise
		power
		\[
		|u|^a=|u(x)|^a;
		\]
		
		\item \(o_n(1)\) denotes a real sequence such that
		\[
		o_n(1)\longrightarrow0
		\qquad\text{as }n\to+\infty;
		\]
		
		\item \(C,C_1,C_2,\ldots\) denote positive constants whose value may
		change from line to line.
		
	\end{enumerate}
	
	\vspace{0.5cm}

	\section{{Variational Framework }}\label{preliminaries}

	\noindent In this section, we introduce the variational framework for problem \eqref{aa} and discuss some results that will be useful throughout the paper. We start with the following embedding result.
	\begin{lemma}\label{lem2.1}{\cite[Theorem~6.21]{Leoni2023fractional}}
		Let $0<s<1$. Then $H^1(\mathbb{R}^2)$ is continuously embedded into $H^s(\mathbb{R}^2)$.
	\end{lemma}
		\noindent The presence of both local and nonlocal operators in \eqref{aa} naturally leads us to consider the space $H^1(\mathbb{R}^2)$ to be the solution space, with the inner product
		\[
		(u,v)=
		\int_{\mathbb{R}^2}\nabla u\cdot\nabla v\,dx
		+ \ll u, v \gg
		+\int_{\mathbb{R}^2}uv\,dx,
		\]
		where 
		$$\ll u,v\gg:=\int_{\mathbb{R}^2}\int_{\mathbb{R}^2}
		\frac{(u(x)-u(y))(v(x)-v(y))}{|x-y|^{2+2s}}\,dx\,dy$$
		and the associated norm $$\|u\|=(u,u)^{1/2}=\left(\|\nabla u\|_2^2+[u]_s^2+\|u\|_2^2\right)^{\frac{1}{2}},$$
		with $$[u]_s^2=\int_{\mathbb{R}^2}\int_{\mathbb{R}^2}
		\frac{|u(x)-u(y)|^2}{|x-y|^{2+2s}}\,dx\,dy.$$
		Next, we recall the Trudinger-Moser inequality in $\mathbb{R}^2$ that will help us to deal with the exponential growth of $f$.
		\begin{proposition}\label{PRa}\cite{Cao,Cassani2014} 
			If $\gamma>0$ and $u\in H^{1}(\mathbb{R}^{2})$, then
			\begin{eqnarray*}
				\begin{aligned}
					\int_{\mathbb{R}^{2}}\Big(e^{\gamma |u|^{2}}-1\Big)dx<\infty.
				\end{aligned}
			\end{eqnarray*}
			Moreover, if $\gamma<4\pi$ and $\|u\|_{2}\leq M<+\infty$, then there exists a constant $\mathcal{C}(M,\gamma)>0$ such that
			\begin{eqnarray*}
				\begin{aligned}
					\sup_{\|\nabla u\|^{2}_{2}\leq1, \|u\|_{2}\leq M }
					\int_{\mathbb{R}^{2}}\Big(e^{\gamma |u|^{2}}-1\Big)dx
					<\mathcal{C}(M,\gamma).
				\end{aligned}
			\end{eqnarray*}
		\end{proposition}

		\noindent We now introduce and prove a Trudinger-Moser type inequality for the mixed local-nonlocal setting.

		\begin{proposition}\label{mixedTM}
			Let $u\in 
			H^{1}(\mathbb{R}^{2})$.
			Then the following assertions hold:
			
			\begin{enumerate}
				\item[(i)] If $\gamma<4\pi$ and $\|u\|\leq 1$, then there exists a constant
				$C_{\gamma}>0$, depending only on $\gamma$, such that
				\[
				\int_{\mathbb{R}^{2}}\left(e^{\gamma |u|^{2}}-1\right)\,dx\leq C_{\gamma}.
				\]
				
				\item[(ii)] More generally, if $u\in H^1(\mathbb{R}^{2})$ satisfies
				\[
				\gamma\|u\|^{2}<4\pi,
				\]
				then there exists a constant $C>0$ such that
				\[
				\int_{\mathbb{R}^{2}}\left(e^{\gamma |u|^{2}}-1\right)\,dx\leq C.
				\]
			\end{enumerate}
		\end{proposition}
		
		\begin{proof}
			If \(\gamma<4\pi\) and \(\|u\|\le 1\), then
			\[
			\|\nabla u\|_2\le \|u\|\le 1, 
			\qquad
			\|u\|_2\le \|u\|\le 1.
			\]
			Thus Proposition~\ref{PRa} gives
			\[
			\int_{\mathbb R^2}\left(e^{\gamma |u|^2}-1\right)\,dx\le C_\gamma,
			\]
			and hence (i) follows. Finally, assume that \(\gamma\|u\|^2<4\pi\). If \(u=0\), the result is trivial.
			Otherwise set \(v=u/\|u\|\). Then \(\|v\|=1\). Let
			\[
			\beta:=\gamma\|u\|^2<4\pi.
			\]
			Applying (i) to \(v\) with exponent \(\beta\), we obtain
			\[
			\int_{\mathbb R^2}\left(e^{\beta |v|^2}-1\right)\,dx\le C_\beta.
			\]
			Since \(\gamma |u|^2=\beta |v|^2\), we get
			\[
			\int_{\mathbb R^2}\left(e^{\gamma |u|^2}-1\right)\,dx
			=
			\int_{\mathbb R^2}\left(e^{\beta |v|^2}-1\right)\,dx
			\le C_\beta.
			\]
			This proves (ii).
		\end{proof}
		\noindent We also recall the Hardy-Littlewood-Sobolev inequality; see \cite{Lieb}.
		
		\begin{proposition}\label{PRc}
			Let $t_1,t_2>1$, $0<\alpha<2$, with $\frac{1}{t_1}+\frac{2-\alpha}{2}+\frac{1}{t_2}=2$. If $f\in L^{t_1}(\mathbb{R}^{2})$ and $g\in L^{t_2}(\mathbb{R}^{2})$, then there exists a constant $C(t_1,\alpha,t_2)>0$ such that
			\begin{eqnarray*}
				\begin{aligned}
					\int_{\mathbb{R}^{2}} (I_{\alpha}\ast f)g\, dx \leq C(t_1,\alpha,t_2)\|f\|_{t_1}\|g\|_{t_2}.
				\end{aligned}
			\end{eqnarray*}
		\end{proposition}
		\noindent As a consequence of Proposition \ref{PRc}, the term
		\[
		\mathcal{A}_{F}(u):=\int_{\mathbb{R}^{2}} (I_{\alpha}\ast F(u))F(u)\, dx
		\]
		is well defined provided that $F(u)
		\in
		L^{\frac{4}{2+\alpha}}(\mathbb{R}^{2})$.			
		\noindent Finally, we introduce the notion of weak solutions.
		\begin{definition}\label{def1.1}
			A function $u \in H^1(\mathbb{R}^2)$ is said to be a weak solution of \eqref{aa}
			if $\|u\|_2^2 = a^2$ and
			\[
			\int_{\mathbb{R}^2} \nabla u \cdot \nabla v\,dx
			+\ll u,v \gg
			+\lambda \int_{\mathbb{R}^2}uv\,dx
			=
			\Lambda\int_{\mathbb{R}^2}
			(I_\alpha * F(u))\,f(u)v\,dx,
			\]
			for all $v \in H^1(\mathbb{R}^2)$.
		\end{definition}
		\noindent Solutions of \eqref{aa}, called the normalized solutions, correspond to critical points of the energy functional
		$
		J:H^1(\mathbb{R}^{2}) \to \mathbb{R}
		$
		defined by
		\begin{equation}\label{J}
			\begin{aligned}
				J(u)
				&=\frac{\|\nabla u\|_2^2}{2}
				+\frac{[u]^2_s}{2} -\frac{\Lambda}{2}\mathcal{A}_F(u), 
			\end{aligned}
		\end{equation}
		constrained to the following $L^2$- sphere
		\begin{equation}\label{tau}
			\mathcal{S}(a):=\left\{u\in H^{1}(\mathbb{R}^{2}) : \|u\|_2^2 = a^2\right\}.
		\end{equation}
		For $a>0$, if $u$ is a critical point of the constrained functional
		$J|_{\mathcal{S}(a)}$, i.e.,
		\(
		J'|_{\mathcal{S}(a)}(u)=0,
		\)
		then, since $\mathcal{S}(a)$ is regular, there exists some $\lambda\in\mathbb R$
		such that $(u,\lambda)$ solves \eqref{aa}. Here, the parameter $\lambda\in\mathbb{R}$ will appear as a Lagrange multiplier depending on the solution
		$u\in H^{1}(\mathbb R^2)$ and is not a priori given. 
		To check that the functional $J$ is well defined,
		it is enough to see if $(I_{\alpha}*F(u))F(u)\in L^1(\mathbb{R}^2)$.
		Since, by using Proposition \ref{PRc}, \eqref{nona} and
		H$\ddot{\text{o}}$lder's inequality, for any $t,t'>1$ satisfying $\frac{1}{t}+\frac{1}{t'}=1$, we have:
		\begin{eqnarray*}
			\int_{\mathbb{R^2} }(I_{\alpha}*F(u))F(u) & \leq & C(\alpha) \left(\int_{\mathbb{R}^{2}}|F(u)|^{\frac{4}{2+\alpha}}\right)^{\frac{2+\alpha}{4}}\left(\int_{\mathbb{R}^{2}}|F(u)|^{\frac{4}{2+\alpha}}\right)^{\frac{2+\alpha}{4}}\\
			& = & C(\alpha) \left(|F(u)|^{\frac{4}{2+\alpha}}\right)^{\frac{2+\alpha}{2}}\\
			& \leq & C(\alpha) \left(\int_{\mathbb{R}^{2}}\left(\eps |u|^{2+\frac{\alpha}{2}}+C_{\eps}|u|^q\left(e^{\gamma|u|^2}-1\right)\right)^{\frac{4}{2+\alpha}}\right)^{\frac{2+\alpha}{2}}\\
			& \leq & C'(\alpha)\left(\left(\int_{\mathbb{R}^{2}}\eps^{\frac{4}{2+\alpha}}|u|^{\left(2+\frac{\alpha}{2}\right)\frac{4}{2+\alpha}}\right)^{\frac{2+\alpha}{2}} \right.\\
			&&+ \left.\left(\int_{\mathbb{R}^{2}}C_{\eps}^{\frac{4}{2+\alpha}}|u|^{\frac{4q}{2+\alpha}}\left(e^{\gamma|u|^2}-1\right)^{\frac{4}{2+\alpha}}\right)^{\frac{2+\alpha}{2}}\right)\\
			& \leq & C'(\alpha) \left(\eps^2 \|u\|_{\left(2+\frac{\alpha}{2}\right)\frac{4}{2+\alpha}}^{2\left(2+\frac{\alpha}{2}\right)}+C_{\eps}^2\left(\int_{\mathbb{R}^{2}}|u|^{\frac{4qt'}{2+\alpha}}\right)^{\frac{2+\alpha}{2t'}}\left(\int_{\mathbb{R}^{2}}\left(e^{\gamma|u|^2}-1\right)^{\frac{4t}{2+\alpha}}\right)^{\frac{2+\alpha}{2t}}\right)\\
			& < & +\infty,
		\end{eqnarray*}
		by the continuous inclusion of $H^1(\mathbb{R}^2)\hookrightarrow L^r(\mathbb{R}^2)$ for all $r\in [2,\infty)$, the fact that
		$$(e^s-1)^m \le C_m (e^{ms}-1), \qquad ~ m>1,~s\ge 0,$$
		and  using the Proposition \ref{PRa}. Therefore, $J$ is well defined and of class $C^{1}$. 
		However, it is well known that $J$ does not satisfy the Palais-Smale condition. 
		To deal with this lack of compactness, we exploit the associated Poho\v{z}aev identity.
		
		\noindent We will be following the work of \cite[Theorem~2.5]{Anthal2025Pohozaev} for the case of Choquard nonlinearity as done in \cite[Theorem~5.0.2]{Anthal2026Ground}; and for that, we need any weak solution to be H$\ddot{\text{o}}$lder continuous. Let us start with studying the regularity of the solution.
		\begin{theorem}\label{Holder_regularity}
			Let $u\in H^1(\mathbb{R}^2)$ be a weak solution of:
			$$-\Delta u +(-\Delta)^s u +\lambda u = \Lambda(I_{\alpha}*F(u))f(u) \text{ in } \mathbb{R}^2,$$
			then $u\in C^{1,\delta}_{loc}(\mathbb{R}^2)$, for all $0<\delta<1$. 
		\end{theorem}
		\begin{proof}
			Due to the embedding results, precisely \cite[Theorem~2.5.2]{Kesavan2019Topics}, it is very well known that $u\in L^t(\mathbb{R}^2)$ for all $t\geq 2$. Next, we claim that $(I_{\alpha}*F(u))\in L^{\infty}(\mathbb{R}^2)$.\\
			Now, for any $x\in \mathbb{R}^2$, by \eqref{nona} we have:
			\begin{eqnarray*}
				(I_{\alpha}*F(u))(x) & = & \int_{\mathbb{R}^2} \frac{A_{\alpha}F(u(y))}{|x-y|^{2-\alpha}}dy = \int_{\mathbb{R}^2} \frac{A_{\alpha}F(u(x-y))}{|y|^{2-\alpha}}dy\\
				& \leq & A_{\alpha}\eps \int_{\mathbb{R}^2}\frac{|u(x-y)|^{2+\frac{\alpha}{2}}}{|y|^{2-\alpha}}dy + A_{\alpha}\kappa_{\eps} \int_{\mathbb{R}^2}\frac{|u(x-y)|^q\left(e^{\gamma|u(x-y)|^2}-1\right)}{|y|^{2-\alpha}}dy,
			\end{eqnarray*}
			with $q>1$.
			Denote $$I_1=\int_{\mathbb{R}^2}\frac{|u(x-y)|^{2+\frac{\alpha}{2}}}{|y|^{2-\alpha}}dy; \text{ and } I_2=\int_{\mathbb{R}^2}\frac{|u(x-y)|^q\left(e^{\gamma|u(x-y)|^2}-1\right)}{|y|^{2-\alpha}}dy.$$
			For any fixed $\beta> \frac{2}{\alpha}>1$, we take $p\in \left(1,\frac{2(\beta-1)}{\beta(2-\alpha)}\right)$. Then, denoting $B_1$ to be the ball centered at origin, we have:
			\begin{equation*}
				I_1  =  \int_{\mathbb{R}^2}\frac{|u(x-y)|^{2+\frac{\alpha}{2}}}{|y|^{2-\alpha}}dy=\int_{B_1}\frac{|u(x-y)|^{2+\frac{\alpha}{2}}}{|y|^{2-\alpha}}dy +\int_{\mathbb{R}^2\setminus B_1}\frac{|u(x-y)|^{2+\frac{\alpha}{2}}}{|y|^{2-\alpha}}dy,
			\end{equation*}
			here,
			\begin{equation*}
				\int_{\mathbb{R}^2\setminus B_1}\frac{|u(x-y)|^{2+\frac{\alpha}{2}}}{|y|^{2-\alpha}}dy \leq \int_{\mathbb{R}^2\setminus B_1}|u(x-y)|^{2+\frac{\alpha}{2}}dy \leq \int_{\mathbb{R}^2}|u(y)|^{2+\frac{\alpha}{2}}dy<+\infty;
			\end{equation*}
			and by H$\ddot{\text{o}}$lder's inequality we have:
			\begin{eqnarray*}
				\int_{B_1}\frac{|u(x-y)|^{1+\frac{\alpha}{2}}}{|y|^{2-\alpha}}dy & \leq &\left(\int_{B_1}|u(x-y)|^{\beta(2+\frac{\alpha}{2})}\right)^{\frac{1}{\beta}}\left(\int_{B_1} \frac{dy}{|y|^{(2-\alpha)\frac{\beta}{\beta-1}}}\right)^{\frac{\beta-1}{\beta}} \\
				& \le & K_1 \left( \int_{\mathbb{R}^2}|u|^{\beta(2+\frac{\alpha}{2})}\right)^{\frac{1}{\beta}}<+\infty,
			\end{eqnarray*}
			for $\beta>\frac{2}{\alpha}$. Now, using Proposition \ref{mixedTM} and H$\ddot{\text{o}}$lder's inequality, we estimate $I_2$ as follows:
			\begin{eqnarray*}
				I_2 & = & \int_{\mathbb{R}^2}\frac{|u(x-y)|^q\left(e^{\gamma|u(x-y)|^2}-1\right)}{|y|^{2-\alpha}}dy\\
				& \leq & \left(\int_{\mathbb{R}^2}|e^{\gamma|u(x-y)|^2}-1|^{\frac{p}{p-1}}\right)^{\frac{p-1}{p}} \left(\int_{\mathbb{R}^2}\frac{|u(x-y)|^{qp}}{|y|^{p(2-\alpha)}}dy\right)^{\frac{1}{p}}\\
				& = & \left(\int_{\mathbb{R}^2}|e^{\gamma|u(x)|^2}-1|^{\frac{p}{p-1}}\right)^{\frac{p-1}{p}} \left(\int_{\mathbb{R}^2}\frac{|u(x-y)|^{qp}}{|y|^{p(2-\alpha)}}dy\right)^{\frac{1}{p}} \le K_2 \left(\int_{\mathbb{R}^2}\frac{|u(x-y)|^{qp}}{|y|^{p(2-\alpha)}}dy\right)^{\frac{1}{p}}\\
				& = & K_2 \left( \int_{B_1}\frac{|u(x-y)|^{qp}}{|y|^{p(2-\alpha)}}dy+\int_{\mathbb{R}^2\setminus B_1}\frac{|u(x-y)|^{qp}}{|y|^{p(2-\alpha)}}dy\right)^{\frac{1}{p}}\\
				& \le & K_2 \left(\int_{B_1}\frac{|u(x-y)|^{pq}}{|y|^{p(2-\alpha)}}dy+\int_{\mathbb{R}^2} |u(x)|^{pq}dx \right)^{\frac{1}{p}}\\
				& \le & K_3 \left(1+ \left(\int_{B_1}\frac{|u(x-y)|^{pq}}{|y|^{p(2-\alpha)}}\right)^{\frac{1}{p}}\right)\\
				& \leq & K_3 \left(1+ \left(\int_{B_1}|u(x-y)|^{pq\beta}\right)^{\frac{1}{p\beta}}\left(\int_{B_1}\frac{dy}{|y|^{(2-\alpha)\frac{p \beta}{\beta-1}}}\right)^{\frac{(\beta-1)}{p\beta}}\right)<+\infty,
			\end{eqnarray*}
			for $p<\frac{2(\beta-1)}{\beta(2-\alpha)}$. Thus $(I_{\alpha}*F(u))\in L^{\infty}(\mathbb{R}^2)$. Now, let us reformulate our problem as:
			$$-\Delta u +\mathcal{L}_1 u +\lambda u = \Lambda(I_{\alpha}*F(u))f(u)-\mathcal{L}_2u;$$
			where $$\mathcal{L}_1 u(x) = P.V \int_{|x-y|\leq B}\frac{u(x)-u(y)}{|x-y|^{2+2s}}dy; $$
			and,
			$$ \mathcal{L}_2u(x) =P.V \int_{|x-y|>B}\frac{u(x)-u(y)}{|x-y|^{2+2s}}dy,$$
			for some fixed $B>0$.
			Define $g(x)= \Lambda(I_{\alpha}*F(u))(x)f(u)(x)-\mathcal{L}_2(u)(x)$, then by Jensen's inequality, above claim, \eqref{nonb} and the fact that $u\in L^t (\mathbb{R}^2)$, for all $t\ge 2$, we get
			\begin{eqnarray*}
				\int_{\mathbb{R}^2}|g(x)|^t & \le & K_4\left(\Lambda\int_{\mathbb{R}^2}|(I_{\alpha}*F(u))f(u)|^t+\int_{\mathbb{R}^2} |\mathcal{L}_2(u)|^t\right)\\
				& \le & K_5 \left(\Lambda\int_{\mathbb{R}^2}|f(u)|^t +\int_{\mathbb{R}^2}\left(\int_{|x-y|>B}\frac{|u(x)-u(y)|^t}{|x-y|^{t(2+2s)}}dy\right)dx\right)\\
				& \le & K_6 \left(\eps\Lambda\int_{\mathbb{R}^2}|u|^{t(1+\frac{\alpha}{2}))}+C_{\eps}\int_{\mathbb{R}^2}|u|^{(q-1)t}\left(e^{\gamma t |u|^2}-1\right)\right.\\
				&& \left.+2 \int_{\mathbb{R}^2}|u(x)|^t\left(\int_{|x-y|>B} \frac{dy}{|x-y|^{t(2+2s)}}\right)dx\right)\\
				& \le & K_7 \left(\eps\Lambda \int_{\mathbb{R}^2}|u(x)|^{t(1+\frac{\alpha}{2})}+C_\eps \left(\int_{\mathbb{R^2}} |u|^{2t(q-1)} \right)^{\frac{1}{2}}\left(\int_{\mathbb{R}^2}\left(e^{\gamma t |u|^2}-1\right)^2\right)^{\frac{1}{2}}\right.\\
				&& \left. +2\int_{\mathbb{R}}^2|u|^t\right)<+\infty \text{ for all } t\geq 2.
			\end{eqnarray*}
			Thus $g\in L^t(\mathbb{R}^2)$ for all $t\ge 2$. Now, by \cite[Theorem~3.1.20]{Garroni2002Second} with $\Omega_I$ being the ball of radius $B$ centered at origin and a fixed bounded domain $\Omega$, get $u\in W^{2,t}_{loc}(\mathbb{R}^2)$ for all $t\geq 2$. Further using Sobolev inequality \cite[Theorem~2.5.4]{Kesavan2019Topics}, we conclude that $u \in C^{1,\delta}_{\text{loc}}(\mathbb R^2)$, for all $\delta \in (0,1)$.
		\end{proof}

		\noindent Moreover, we will need the following result in order to get our Poho\v{z}aev identity.
		{
			\begin{proposition}\label{prop_poho}
				Let \(u\in H^1(\mathbb R^2)\) and \(\lambda\in\mathbb R\) be a weak
				solution of
				\[
				\mathcal Lu+\lambda u
				=
				\Lambda(I_\alpha*F(u))f(u)\text{ in }\mathbb R^2.
				\]
				If \(\widetilde u\) lies between \(u\) and \(u(\cdot+he_j)\), for any $h>0$ and unit vector $e_j$. Then
				\(
				f(\widetilde u)\to f(u)
				~~\text{in }L^2(\mathbb R^2).
				\)
			\end{proposition}
		}
		\begin{proof}
			Now, since $f$ is continuous and $\tilde{u}\rightarrow u$ a.e. in $\mathbb{R}^2$ as $h\rightarrow 0$, we get:
			$$f(\tilde{u})\rightarrow f(u) \text{ a.e. in } \mathbb{R}^2.$$
			We will be using Vitali's convergence theorem in order to prove the convergence. In that context, firstly, we need uniform integrability, that is, for any $\eta>0$ we should be able to find a $\delta>0$ independent of $f(\tilde{u})$, such that:
			$$\sup_{|h|\leq h_0} \int_{E} |f(\tilde{u})|^2 <\eta \text{ whenever } |E|<\delta. $$
			Now, by \eqref{nonb}  and the fact that 
			$$|\tilde{u}(x)|\leq |u(x)|+|u(x+he_j)| \text{ for all } x\in \mathbb{R}^2,$$
			we have:
			\begin{eqnarray*}
				|f(\tilde{u})|^{2r} & \leq & C_r \eps^{2r} |\tilde{u}|^{2r(1+\frac{\alpha}{2})} +C_r'C_{\eps}^{2r}|\tilde{u}|^{2r(q-1)}\left(e^{2r\gamma|\tilde{u}|^2}-1\right)\\
				& \leq & C_r \eps^{2r} |{u}+u(.+he_j)|^{2r(1+\frac{\alpha}{2})} + C_r'C_{\eps}^{2r}|{u}+u(.+he_j)|^{2r(q-1)}\left(e^{2r\gamma|{u}|^2}e^{2r\gamma u(.+he_j)|^2}-1\right)\\
				& \leq & C_r\eps^{2r}C_1|u|^{2r(1+\frac{\alpha}{2})} +C_rC_1\eps^{2r}|u(.+he_j)|^{2r(1+\frac{\alpha}{2})}\\ && +C_r'C_{\eps}^{2r}\left(|u|^{2r(q-1)}+|u(.+he_j)|^{2r(q-1)}\right)\left(e^{2r\gamma|{u}|^2}e^{2r\gamma u(.+he_j)|^2}-1\right),
			\end{eqnarray*}
			further since $(a-1)(b-1)+a-1+b-1=ab-1$, we get:
			\begin{eqnarray*}
				|f(\tilde{u})|^{2r} & \leq & C_r\eps^{2r}C_1|u|^{2r(1+\frac{\alpha}{2}))} +C_rC_1\eps^{2r}|u(.+he_j)|^{2r(1+\frac{\alpha}{2}))}\\ && +C_r'C_2C_{\eps}^{2r}\left(|u|^{2r(q-1)}+|u(.+he_j)|^{2r(q-1)}\right)\left(\left(e^{2r\gamma|{u}|^2}-1\right)\left(e^{2r\gamma u(.+he_j)|^2}-1\right)\right.\\
				&& \left.+\left(e^{2r\gamma|{u}|^2}-1\right)+\left(e^{2r\gamma u(.+he_j)|^2}-1\right)\right).
			\end{eqnarray*}
			Thus, by H$\ddot{\text{o}}$lder's inequality and Proposition \ref{PRa}
			\begin{eqnarray*}
				\int_{\mathbb{R}^2} |f(\tilde{u})|^{2r} & \leq & 2C_r\eps^{2r}C_1 \| u \|_{2r(1+\frac{\alpha}{2})}^{2r(1+\frac{\alpha}{2})} +C_r'C_2C_{\eps}^{2r}\int_{\mathbb{R}^2}|u|^{2r(q-1)}\left(e^{2r\gamma|{u}|^2}-1\right)\left(e^{2r\gamma u(.+he_j)|^2}-1\right)\\
				&&+ C_r'C_2C_{\eps}^{2r}\int_{\mathbb{R}^2}|u(.+he_j)|^{2r(q-1)}\left(e^{2r\gamma|{u}|^2}-1\right)\left(e^{2r\gamma u(.+he_j)|^2}-1\right)\\
				&& + C_r'C_2C_{\eps}^{2r}\int_{\mathbb{R}^2}|u|^{2r(q-1)}\left(e^{2r\gamma|{u}|^2}-1\right)\\
				&& + C_r'C_2C_{\eps}^{2r}\int_{\mathbb{R}^2}|u|^{2r(q-1)}\left(e^{2r\gamma|{u(.+he_j)}|^2}-1\right)\\
				&& + C_r'C_2C_{\eps}^{2r}\int_{\mathbb{R}^2}|u(.+he_j)|^{2r(q-1)}\left(e^{2r\gamma|{u}|^2}-1\right)\\
				&& + C_r'C_2C_{\eps}^{2r}\int_{\mathbb{R}^2}|u(.+he_j)|^{2r(q-1)}\left(e^{2r\gamma|{u(.+he_j)}|^2}-1\right)\\
				& \leq & C_r'\eps^{2r}C_1 \|u\|_{2r(1+\frac{\alpha}{2})}^{2r(1+\frac{\alpha}{2})} + C_3C_{\eps}^2| u |_{4r(q-1)}^{2r(q-1)}\int_{\mathbb{R}^2}\left(e^{8r\gamma|u|^2}-1\right) \\
				&&+ C_4C_{\eps}^2 \| u\|_{4r(q-1)}^{2r(q-1)} \left(\int_{\mathbb{R}^2}\left(e^{4r\gamma|u|^2}-1\right)\right) <+\infty \text{ for all } |h|\leq h_0,
			\end{eqnarray*} 
			Hence 
			\begin{equation}\label{ebb}
				\sup_{|h|\leq h_0}
				\int_{\mathbb R^2}|f(\widetilde u)|^{2r}\,dx<+\infty.
			\end{equation}
			Let $\eta>0$.
			For any measurable set $E\subset \mathbb R^2$, by H\"older's inequality and \eqref{ebb},
			\[
			\int_E |f(\widetilde u)|^2\,dx
			\leq
			\left(\int_E |f(\widetilde u)|^{2r}\,dx\right)^{1/r}
			|E|^{(r-1)/r}
			\leq C |E|^{(r-1)/r}.
			\]
			Hence, choosing $\delta>0$ such that $C\delta^{(r-1)/r}<\eta$, we obtain
			\[
			|E|<\delta
			\quad\Longrightarrow\quad
			\sup_{|h|\leq h_0}\int_E |f(\widetilde u)|^2\,dx<\eta.
			\]
			Next, we need the tightness property, that is, for any $R>0$, we should be able to find $S_{R}\subset \mathbb{R}^2$ such that $|S_{R}|<+\infty$ and 
			\begin{equation}\label{tightness}
				\sup \int_{\mathbb{R}^2\setminus S_R}|f(\tilde{u})|^2 \rightarrow 0 \text{ as } h\rightarrow 0.
			\end{equation}
			Now, for any ball $B_{R}(0)$ of radius $R$ centered at origin, if we denote $A_R:=\mathbb{R}^2\setminus B_R(0)$, then we have
			\begin{eqnarray}\label{eq_1}
				\int_{A_{R}} |f(\tilde{u})|^2 & \leq & 2 \eps\int_{A_R}|\tilde{u}|^{2(1+\frac{\alpha}{2})} +4C_{\eps} \int_{A_{R}}|\tilde{u}|^{2(q-1)}\left(e^{2\gamma|\tilde{u}|^2}-1\right)\nonumber\\
				& \leq & 2\eps \int_{A_{R}}|\tilde{u}|^{2(1+\frac{\alpha}{2})} +C_5 \left(\int_{A_R}|\tilde{u}|^{4(q-1)}\right)^{\frac{1}{2}}\left(\int_{A_{R}}\left(e^{4\gamma|\tilde{u}|^2}-1\right)\right)^{\frac{1}{2}}.
			\end{eqnarray}
			Further, since $\tilde{u}$ lies in $L^p(\mathbb{R}^2)$ for all $p\geq 2$, by Dominated convergence theorem,
			$|\tilde{u}|^p\chi_{\mathbb{R}^2\setminus B_R(0)}\rightarrow 0$ as $R\rightarrow \infty$ in $L^1(\mathbb{R}^2)$, that is 
			\begin{equation*}
				\int_{A_R}|\tilde{u}|^p=\int_{\mathbb{R}^2}|\tilde{u}|^p\chi_{\mathbb{R}^2\setminus B_R(0)}dx \rightarrow 0 \text{ as } R\rightarrow \infty \text{ for all }p\geq 2 \text{ and } |h|\leq h_0,
			\end{equation*}
			thus,
			\begin{equation}\label{eq_2}
				\sup_{|h|\leq h_0} \int_{A_R}|\tilde{u}|^p \rightarrow 0 \text{ as } R\rightarrow \infty \text{ for all } p\geq2;
			\end{equation}
			also, as done above in the uniform integrability case, we can find $M>0$ such that 
			\begin{equation}\label{eq_3}
				\int_{A_R}\left(e^{4\gamma|\tilde{u}|^2}-1\right)\leq M \text{ for all } h.
			\end{equation}
			Therefore, using \eqref{eq_2} and \eqref{eq_3} in \eqref{eq_1} we get \eqref{tightness} with $S_R=B_R(0)$. Hence, by Vitali's convergence theorem, $f(\tilde{u})\rightarrow f(u)$ in $L^2(\mathbb{R}^2)$.
		\end{proof}
		\noindent In order to get the existence of the normalized solution, we had constrtucted the Poho\v{z}aev Manifold using the following Poho\v{z}aev identity:
		{
			\begin{theorem}\label{Pohozaev Indentity}
				Let \(u\in H^1(\mathbb R^2)\) and \(\lambda\in\mathbb R\) be a weak
				solution of
				\[
				\mathcal Lu+\lambda u
				=
				\Lambda(I_\alpha*F(u))f(u)
				\qquad\text{in }\mathbb R^2.
				\]
				Then
				\[
				(1-s)[u]_s^2
				+\lambda\|u\|_2^2
				=
				\Lambda\left(\frac{2+\alpha}{2}\right)\mathcal A_F(u).
				\]
			\end{theorem}
		}
		
		\begin{proof}
			The proof is directly followed from \cite[Theorem~5.0.2]{Anthal2026Ground}, thus, we omit the repetitive calculations and highlight the main difference that occurred due to the growth of the function $f$. Considering the notations as in \cite[Theorem~5.0.2]{Anthal2026Ground} we 
			estimate $I_{3,2}+J_{2,2}$:\\
			\noindent Now, using $\displaystyle \phi_1=\sum_{i=1}^2 D_j(u)$ as test function, we get:
			\begin{equation}\label{phi_1_test}
				\int_{\mathbb{R}^2}\nabla u \nabla \phi_1 +\ll u, \phi_1\gg = \Lambda\int_{\mathbb{R}^2}(I_{\alpha}*F(u))f(u)\phi_1 -\lambda\int_{\mathbb{R}^2}u\phi_1.
			\end{equation}
			Then, as done in \cite[Theorem~2.5]{Anthal2025Pohozaev}, 
			since $$D_j(|\nabla u|^2)-2\nabla u \nabla D_j(\nabla u)\geq 0;$$
			$$D_j(|u(x)-u(y)|^2)-2(u(x)-u(y))(D_ju(x)-D_ju(y))\geq 0 \text{ for all } x, y \in \mathbb{R}^2,$$
			and $x_j\psi_R(x)\leq \frac{2}{R}$ for all $x\in \mathbb{R}^2$,
			we get:
			\begin{eqnarray*}
				|I_{3,2}+J_{2,2}| & = & \frac{1}{2} \left|\sum_{j=1}^2\int_{\mathbb{R}^2}\left(D_j(|\nabla u|^2)-2\nabla uD_j(\nabla u)\right)\psi_Rx_j dx \right.\\
				&& \left. +\sum_{j=1}^2\frac{1}{2}\int_{\mathbb{R}^2}\int_{\mathbb{R}^2} \frac{\left(D_j(|u(x)-u(y)|^2)-2(u(x)-u(y))(D_j u(x)-D_j u(y))\right)\psi_Rx_j}{|x-y|^{2+2s}}dxdy\right|\\
				& \leq & \frac{1}{R} \left|\sum_{j=1}^2\int_{\mathbb{R}^2}\left(D_j(|\nabla u|^2)-2\nabla uD_j(\nabla u)\right)dx  \right.\\
				&& \left. +\sum_{j=1}^2\frac{1}{2}\int_{\mathbb{R}^2}\int_{\mathbb{R}^2} \frac{\left(D_j(|u(x)-u(y)|^2)-2(u(x)-u(y))(D_j u(x)-D_j u(y))\right)}{|x-y|^{2+2s}}dxdy\right|\\
				& = &\frac{1}{R}\sum_{j=1}^2\left(\int_{\mathbb{R}^2}\left(D_j(|\nabla u|^2)-2\nabla uD_j(\nabla u)\right)dx\right.\\
				&& \left. +\frac{1}{2}\int_{\mathbb{R}^2}\int_{\mathbb{R}^2} \frac{\left(D_j(|u(x)-u(y)|^2)-2(u(x)-u(y))(D_j u(x)-D_j u(y))\right)}{|x-y|^{2+2s}}dxdy\right)\\
				& = & -\frac{2}{R}\left(\int_{\mathbb{R}^2}\nabla u \sum_{j=1}^2D_j(\nabla u)+\sum_{i=1}^2 \ll u, D_j(u)\gg\right)\\
				& = & -\frac{2}{R} \left(\int_{\mathbb{R}^2}\nabla u \nabla \phi_1+\ll u,\phi_1 \gg\right)
				=  \frac{2}{R}\left(0-\int_{\mathbb{R}^2}(\Lambda(I_{\alpha}*F(u))f(u)-\lambda u)\phi_1\right).
			\end{eqnarray*}
			Now, since $\Lambda(I_{\alpha}*F(u))F(u)-\lambda u\in L^1(\mathbb{R}^2)$ we get:
			$$\int_{\mathbb{R}^2}D_j\left(\Lambda(I_{\alpha}*F(u))F(u)-\frac{\lambda |u|^2}{2}\right)dx=0,$$
			since for any integrable function $g$, we have:
			$$\int_{\mathbb{R}^2}D_j(g)(x)dx= \frac{1}{h}\int_{\mathbb{R}^2}\left(g(x+he_j)-g(x)\right)dx = \frac{1}{h}\left( \int_{\mathbb{R}^2}g(x)dx-\int_{\mathbb{R}^2}g(x)dx\right)=0;$$
			and hence, 
			\begin{eqnarray}\label{I_3,2_J_2,2}
				|I_{3,2}+J_{2,2}| & \leq & \frac{2}{R}\left(\int_{\mathbb{R}^2}\sum_{j=1}^2 D_j\left(\Lambda(I_{\alpha}*F(u))F(u)-\frac{|u|^2}{2}\right)dx-\int_{\mathbb{R}^2}(\Lambda(I_{\alpha}*F(u))f(u)- \lambda u)\phi_1\right)\nonumber\\
				& = & \frac{2}{R} \sum_{j=1}^2 \int_{\mathbb{R}^2}\left(D_j\left(\Lambda(I_{\alpha}*F(u))F(u)-\frac{|u|^2}{2}\right)-(\Lambda(I_{\alpha}*F(u))f(u)-\lambda u)D_j(u)\right)dx\nonumber\\
				& = & \frac{2}{R}\left(\sum_{j=1}^2 A^1_j - A^2_j\right),
			\end{eqnarray}
			where $$A_j^1= \Lambda\int_{\mathbb{R}^2}\left(D_j((I_{\alpha}*F(u))F(u)-(I_{\alpha}*F(u))f(u)D_j(u)\right)dx;$$
			and 
			$$A_j^2=\int_{\mathbb{R}^2}\left(D_j\left(\frac{\lambda |u|^2}{2}\right)-\lambda uD_j(u)\right)dx.$$
			Now, since $(I_{\alpha}*F(u))\in L^{\infty}(\mathbb{R}^N)$ by using mean value theorem, we get
			\begin{eqnarray*}
				A_j^1 & = & \Lambda \int_{\mathbb{R}^2}\left(D_j((I_{\alpha}*F(u))F(u)-(I_{\alpha}*F(u))f(u)D_j(u)\right)dx\\
				& \leq & C\Lambda\int_{\mathbb{R}^2}\left(\frac{F(u)(x+he_j)-F(u)(x)}{h}-f(u)D_j(u)\right)dx\\
				& = & C\Lambda\int_{\mathbb{R}^2}\left(D_j(F(u))-f(u)D_j(u)\right)dx\\
				& = & C\Lambda\int_{\mathbb{R}^2}\left(f(\tilde{u})D_j(u)-f(u)D_j(u)\right)dx,
			\end{eqnarray*}
			where $\tilde{u}$ lies between $u$ and $u(.+he_j)$ Thus.
			\begin{equation}\label{A_j_1}
				A_j^1 \le C\Lambda\int_{\mathbb{R}^2}(f(\tilde{u})-f(u))D_j(u)dx;
			\end{equation}
			also, by the mean value theorem:
			\begin{equation}\label{A_j_2}
				A_j^2 = \int_{\mathbb{R}^2}(\tilde{u}-u)D_j(u).
			\end{equation}
			Therefore, using \eqref{A_j_1} and \eqref{A_j_2} in \eqref{I_3,2_J_2,2}, we get:
			\begin{equation*}
				|I_{3,2}+J_{2,2}| \leq C_1 \left(\sum_{j=1}^2\Lambda\int_{\mathbb{R}^2}(f(\tilde{u})-f(u))D_j(u)dx + \sum_{j=1}^2\int_{\mathbb{R}^2}(\tilde{u}-u)D_j(u)dx\right).
			\end{equation*}
			Here, since $D_j u \rightarrow \frac{\partial u}{\partial x_j}$ and $\tilde{u}\rightarrow u$ as $h\rightarrow 0$, in $L^2(\mathbb{R}^N)$ we get:
			\begin{equation*}
				\sum_{j=1}^2\int_{\mathbb{R}^2}(\tilde{u}-u)D_j(u)dx  \le \sum_{j=1}^{2}\left(\int|\tilde{u}-u|^2\right)^{\frac{1}{2}}\left(\int |D_ju|^2\right)^{\frac{1}{2}} \rightarrow 0 \text{ as } h\rightarrow 0.
			\end{equation*}
			Next, by Proposition \ref{prop_poho} we have
			\begin{eqnarray*}
				\sum_{j=1}^2\int_{\mathbb{R}^2}(f(\tilde{u})-f(u))D_j(u)dx & \le & \sum_{j=1}^2 \left(\int_{\mathbb{R}^2}|f(\tilde{u})-f(u)|^2\right)^{\frac{1}{2}}\left(\int_{\mathbb{R}^2}|D_ju|^2\right)^{\frac{1}{2}}\rightarrow 0,
			\end{eqnarray*}
			as $ h\rightarrow0$. Thus, $|I_{3,2}+I_{2,2}|\rightarrow 0$. Further, following the proof of \cite[Theorem~5.0.2]{Anthal2026Ground} step by step, we get \eqref{Pohozaev_identity}.
		\end{proof}
		
		\begin{lemma}
			If $(u,\lambda)$ is any couple weakly solving problem \eqref{aa}, then $u \in \mathcal{P}(a)$, where
			\begin{equation}\label{Pohozaev-manifold}
				\mathcal{P}(a)
				:=
				\left\{
				u\in \mathcal{S}(a)\;:\; P(u)=0
				\right\},
			\end{equation}
			and
			\begin{equation}\label{Pohozaev}
				\begin{aligned}
					P(u)
					&:=
					\|\nabla u\|_{2}^{2}
					+s[u]_{s}^{2}-\Lambda\tilde{\mathcal{A}}_F(u);\text{ with } \tilde{A}_F(u):= \int_{\mathbb{R}^2}(I_{\alpha}*F(u))\tilde{F}(u)dx.
				\end{aligned}
			\end{equation}
		\end{lemma}
		\begin{proof}
			Indeed, testing the equation with the solution itself, one gets
			\begin{equation}\label{eq_2.5}
				\| \nabla u \|_2^2 + [u]_s^2 +\lambda \| u \|_2^2=\Lambda\int_{\mathbb{R}^N}(I_{\alpha}*F(u))f(u)udx.  
			\end{equation}
			Also, as proved in Theorem \ref{Pohozaev Indentity} any solution to \eqref{aa} must satisfy the following
			\begin{equation}\label{eq_2.6}
				\left(1-s\right)[u]_s^2+\lambda \| u \|_2^2= \Lambda\left(\frac{2+\alpha}{2}\right)\mathcal{A}_F(u).
			\end{equation}
			Using \eqref{eq_2.5} in \eqref{eq_2.6} we get
			\begin{equation*}
				\| \nabla u \|_2^2 +s[u]_s^2+ \Lambda\left(\frac{2+\alpha}{2}\right)\int_{\mathbb{R}^2}(I_{\alpha}*F(u))F(u)-\Lambda\int_{\mathbb{R}^2}(I_{\alpha}*F(u))f(u)u=0.
			\end{equation*}
			\text{that is}, $P(u) = 0$.
		\end{proof}
		\noindent Thus any solution of \eqref{aa} lies on the Poho\v{z}aev manifold $\mathcal{P}(a)$ and we define the ground state solution as follows:
		\begin{definition}\label{defgns}
			Let \(a>0\). 
			A normalized solution \(u_a\in\mathcal S(a)\) is called a
			\emph{normalized ground state solution} of \eqref{aa} if
			\[
			J(u_a)
			=
			\inf\left\{
			J(v):
			v\in\mathcal S(a),\
			J'|_{\mathcal S(a)}(v)=0
			\right\}.
			\]
		\end{definition}
		\noindent Since every normalized solution satisfies the Poho\v{z}aev identity,
		we have
		\[
		\left\{
		u\in\mathcal S(a):
		J'|_{\mathcal S(a)}(u)=0
		\right\}
		\subset\mathcal P(a).
		\]
		We introduce the variational level
		\[
		m(a)
		:=
		\inf_{u\in\mathcal P(a)}J(u).
		\]
		\noindent To overcome the lack of compactness due to whole space $\mathbb{R}^2$, we restrict our analysis to the radial subspace 
		$H^{1}_{rad}(\mathbb{R}^{2})$ of $H^{1}(\mathbb{R}^{2})$. 
		Accordingly, we define
		\[
		\mathcal{S}_{r}(a):=\mathcal{S}(a)\cap H^{1}_{rad}(\mathbb{R}^{2}),
		\]
		where
		\[
		H^{1}_{rad}(\mathbb{R}^{2}) := \left\{ u \in H^{1}(\mathbb{R}^{2}) : u(x) = u(|x|),\ x \in \mathbb{R}^2 \right\}.
		\]
		

		
		\section{{Mountain pass Geometry and Minimax level}}\label{sec3.1}
		
		\noindent In this section, we establish the mountain-pass geometry for the
		constrained functional and introduce the corresponding minimax level
		\(\gamma(a)\).

		\noindent Defining the $L^2$-norm preserving scaling, called the fibre map
		\[
		H(u,\sigma)(x):=e^\sigma u(e^\sigma x), \text{ for every } (u,\sigma)\in H^1(\mathbb{R}^2)\times\mathbb R.
		\]
		Clearly,
		$$\|\nabla H(u,\sigma)\|_2^2=e^{2\sigma}\|\nabla u\|_2^2;\;\;[H(u,\sigma)]_s^2=e^{2s\sigma}[u]_s^2,\;\;\|H(u,\sigma)\|_{\xi}^{\xi}=e^{(\xi-2)\sigma}\|u\|_{\xi}^{\xi} \text{ for any }\xi\geq2$$
		and 
		$$\mathcal{A}_F(H(u,\sigma))=
		e^{-(2+\alpha)\sigma}\mathcal{A}_F(e^\sigma u).
		$$
		We initiate by discussing the behaviour of the constrained functional
		\(J\) along the fibre map.
		
		\begin{lemma}\label{geo_mixed}
			Let \(u\in S_r(a)\) and
			\ref{f0}-\ref{f3} holds. Then \(H(u,\sigma)\in S_r(a)\) for every
			\(\sigma\in\mathbb{R}\), and
			\begin{enumerate}
				\item[(i)] as \(\sigma\rightarrow-\infty\), we have
				\[
				\|\nabla H(u,\sigma)\|_2^2\rightarrow0,\qquad
				[H(u,\sigma)]_s^2\rightarrow0,
				\qquad
				J(H(u,\sigma))\rightarrow0;
				\]
				
				\item[(ii)] as \(\sigma\rightarrow+\infty\), we have
				\[
				\|\nabla H(u,\sigma)\|_2^2\rightarrow+\infty,\qquad
				[H(u,\sigma)]_s^2\rightarrow+\infty,
				\qquad
				J(H(u,\sigma))\rightarrow-\infty.
				\]
			\end{enumerate}
		\end{lemma}
		\begin{proof}
			For every \(\sigma\in\mathbb{R}\), we have
			\[
			\|H(u,\sigma)\|_2^2
			=
			\int_{\mathbb{R}^2}|e^\sigma u(e^\sigma x)|^2\,dx
			=
			\|u\|_2^2
			=
			a^2.
			\]
			Hence,
			\(
			H(u,\sigma)\in S_r(a)
			\text{ for every }\sigma\in\mathbb{R}.
			\)
			Clearly,
			\begin{equation*}
				\|\nabla H(u,\sigma)\|_2^2
				=
				e^{2\sigma}\|\nabla u\|_2^2
				\longrightarrow
				\left\{
				\begin{array}{rl}
					+\infty, & \text{as }\sigma\rightarrow+\infty,\\
					0, & \text{as }\sigma\rightarrow-\infty,
				\end{array}
				\right.
			\end{equation*}
			and
			\begin{equation*}
				[H(u,\sigma)]_s^2
				=
				e^{2s\sigma}[u]_s^2
				\longrightarrow
				\left\{
				\begin{array}{rl}
					+\infty, & \text{as }\sigma\rightarrow+\infty,\\
					0, & \text{as }\sigma\rightarrow-\infty.
				\end{array}
				\right.
			\end{equation*}
			Next, we study the behaviour of \(J(H(u,\sigma))\). Precisely, we have
			\begin{equation}\label{J(H)}
				J(H(u,\sigma))
				=
				\frac{e^{2\sigma}}{2}\|\nabla u\|_2^2
				+
				\frac{e^{2s\sigma}}{2}[u]_s^2
				-\frac{\Lambda}{2}\mathcal{A}_F(H(u,\sigma)).
			\end{equation}
			We first prove \(\textnormal{(i)}\). Set
			\(
			p:=\frac{4}{2+\alpha}.
			\)
			Using Proposition \ref{PRc}, \eqref{nona}, and
			H$\ddot{\text{o}}$lder's inequality, for any \(t,t'>1\) satisfying
			\(\frac{1}{t}+\frac{1}{t'}=1\), we get
			\begin{eqnarray*}
				\mathcal{A}_F(H(u,\sigma))& = & \int_{\mathbb{R}^2}
				(I_{\alpha}*F(H(u,\sigma)))F(H(u,\sigma))\,dx\\
				&\leq&
				C_1
				\left(
				\int_{\mathbb{R}^2}
				|F(H(u,\sigma))|^{\frac{4}{2+\alpha}}\,dx
				\right)^{\frac{2+\alpha}{2}}\\
				&\leq&
				C_1
				\left(
				\int_{\mathbb{R}^2}
				\left(
				\varepsilon |H(u,\sigma)|^{2+\frac{\alpha}{2}}
				+
				\kappa_{\varepsilon}|H(u,\sigma)|^q
				\left(e^{\gamma|H(u,\sigma)|^2}-1\right)
				\right)^{\frac{4}{2+\alpha}}
				dx
				\right)^{\frac{2+\alpha}{2}}\\
				&\leq&
				C_2
				\left(
				\varepsilon^{\frac{4}{2+\alpha}}
				\int_{\mathbb{R}^2}
				\|H(u,\sigma)\|^{\left(2+\frac{\alpha}{2}\right)
					\frac{4}{2+\alpha}}\,dx
				\right.\\
				&&\left.
				\qquad\qquad+
				\kappa_{\varepsilon}^{\frac{4}{2+\alpha}}
				\int_{\mathbb{R}^2}
				|H(u,\sigma)|^{\frac{4q}{2+\alpha}}
				\left(e^{\gamma|H(u,\sigma)|^2}-1\right)^{\frac{4}{2+\alpha}}
				dx
				\right)^{\frac{2+\alpha}{2}}\\
				&\leq&
				C_3\varepsilon^2
				\|H(u,\sigma)\|_{\frac{2\left(4+\alpha\right)}{2+\alpha}}
				^{\left(4+\alpha\right)}\\
				&&+
				C_3\kappa_{\varepsilon}^2
				\left[
				\left(
				\int_{\mathbb{R}^2}
				|H(u,\sigma)|^{\frac{4qt}{2+\alpha}}\,dx
				\right)^{\frac1t}
				\left(
				\int_{\mathbb{R}^2}
				\left(e^{\gamma|H(u,\sigma)|^2}-1\right)^{
					\frac{4t'}{2+\alpha}}
				dx
				\right)^{\frac1{t'}}
				\right]^{\frac{2+\alpha}{2}}.
			\end{eqnarray*}
			Since \(\frac{2(4+ \alpha)}{\alpha+2}>2, \) the first term is well defined.
			Choose \(t'>1\) sufficiently close to \(1\). Since
			\[
			\|\nabla H(u,\sigma)\|_2^2+[H(u,\sigma)]_s^2
			=
			e^{2\sigma}\|\nabla u\|_2^2+e^{2s\sigma}[u]_s^2
			\longrightarrow 0
			\text{ as }\sigma\to-\infty,
			\]
			there exists \(\sigma_0<0\) such that, for every
			\(\sigma\leq\sigma_0\),
			\[
			\frac{4\gamma t'}{2+\alpha}
			\| H(u,\sigma)\|_2^2
			\leq4\pi,
			\]
			thus by Proposition \ref{mixedTM}
			\begin{equation}\label{1Domina1_mixed}
				\begin{aligned}
					&
					\left(
					\int_{\mathbb{R}^{2}}
					\left(e^{\gamma|H(u,\sigma)|^2}-1\right)^{
						\frac{4t'}{2+\alpha}}
					dx
					\right)^{\frac{2+\alpha}{2t'}}\leq
					C_4
					\left(
					\int_{\mathbb{R}^2}
					\left(
					e^{\frac{4t'\gamma}{2+\alpha}|H(u,\sigma)|^2}-1
					\right)dx
					\right)^{\frac{2+\alpha}{2t'}}
					\leq C.
				\end{aligned}
			\end{equation}
			Therefore,
			\begin{eqnarray*}
				\mathcal{A}_F(H(u,\sigma))
				&\leq&
				C_3\varepsilon^2
				\|H(u,\sigma)\|_{\frac{2\left(4+\alpha\right)}{2+\alpha}}
				^{\left(4+\alpha\right)}
				+
				C_5\kappa_{\varepsilon}^2
				\|H(u,\sigma)\|_{\frac{4qt}{2+\alpha}}^{2q}\\
				&=&
				C_3\varepsilon^2e^{2\sigma}
				\|u\|_{\frac{2\left(4+\alpha\right)}{2+\alpha}}
				^{2\left(4+\alpha\right)}
				+
				C_5\kappa_{\varepsilon}^2
				e^{\frac{2qt-2-\alpha}{t}\sigma}
				\|u\|_{\frac{4qt}{2+\alpha}}^{2q}
				\rightarrow 0
				\text{ as }\sigma\rightarrow-\infty,
			\end{eqnarray*}
			since $t'$ is close to $1$, we can choose $t>1$ large enough so that $(2qt-2-\alpha)>0$.
			Therefore, by \eqref{J(H)},
			\(
			J(H(u,\sigma))\rightarrow 0
			\text{ as }\sigma\rightarrow-\infty.
			\)
			This proves \(\textnormal{(i)}\).
			
			\noindent
			We now prove \(\textnormal{(ii)}\). Define
			Set
			\[
			w(t)=\mathcal{A}_F\!\left(\frac{tu}{\|u\|}\right)=\int_{\mathbb{R}^2}\left(I_{\alpha}*F\left(\frac{tu}{\left\| u \right\|}\right)\right)F\left(\frac{tu}{\left\| u \right\|}\right)dx.
			\]
			By \ref{f3}, we know that
			\[
			\frac{w'(t)}{w(t)}\geq \frac{2\mu}{t}
			\qquad \text{for } t>0,
			\]
			which implies that
			\[
			\mathcal{A}_F(tu)\geq \mathcal{A}_F\!\left(\frac{u}{\|u\|}\right)\|u\|^{2\mu}t^{2\mu}.
			\]
			Therefore, taking $t=e^{\sigma}$ by \eqref{J(H)} we obtain
			\begin{align*}
				J(H(u,\sigma))& = \frac{e^{2\sigma}}{2}\| \nabla u \|_2^2+\frac{e^{2s\sigma}}{2}[u]_s^2-\frac{\Lambda}{2e^{(2+\alpha)\sigma}}\mathcal{A}_F(e^{\sigma}u)\\
				&\leq
				C_{1}e^{2\sigma}
				+
				C_{2}e^{2s\sigma}
				-
				C_{3}e^{(2\mu-(2+\alpha))\sigma}.
			\end{align*}
			Since $0<s<1$ and
			\(
			2\mu-(2+\alpha)>2>2s
			\)
			the negative term dominates as $\sigma\to +\infty$. Hence,
			\[
			J(H(u,\sigma))\to -\infty
			\qquad \text{as }\sigma\to +\infty.
			\]
			This proves $(ii)$ and completes the proof.
		\end{proof}

		\noindent
		For every \(k>0\), define
		\[
		A_k
		:=
		\left\{
		u\in\mathcal{S}_r(a):
		T(u)=\|\nabla u\|_2^2+[u]_s^2\leq k
		\right\},
		\]
		and
		\[
		B_k
		:=
		\left\{
		u\in\mathcal{S}_r(a):
		T(u)=\|\nabla u\|_2^2+[u]_s^2=2k
		\right\}.
		\]
		We study the functional \(J\) constrained to these sets.
		
		\begin{lemma}\label{P1_mixed}
			Assume that
			\ref{f0}-\ref{f2} hold.
			There exists $K(a)>0$ such that $J(u)$, $P(u)>0$ for all $u\in A$ and 
			$$0<\sup_{u\in A}J(u)\leq \inf_{u\in B} J(u),$$
			where $A=A_{K(a)}$ and $B=B_{K(a)}$.
		\end{lemma}
		
		\begin{proof} Suppose that \(u\in\mathcal{S}_r(a)\) satisfies \(T(u)=\|\nabla u\|_2^2+[u]_s^2=k. \) 
			Following the arguments of Lemma \ref{geo_mixed}, we can find \(t,t'>1\), satisfying \( \frac{1}{t}+\frac{1}{t'}=1, \) with \(t'\) sufficiently close to 1, such that 
			\begin{align*} 
				\mathcal{A}_F(u)
				&\leq C_1\varepsilon^2 \|u\|_{\frac{2\left(4+\alpha\right)}{2+\alpha}} ^{4+\alpha} + C_2\kappa_\varepsilon^2 \|u\|_{\frac{4qt}{2+\alpha}}^{2q}. \end{align*} Moreover, by the Gagliardo-Nirenberg inequality \cite[Theorem 1.3.7]{Cazenave2003}, 
			\begin{equation}\label{1eq} 
				\|u\|_p \leq C(p)\|\nabla u\|_2^{\theta}\|u\|_2^{1-\theta} \leq C(p) \left( \|\nabla u\|_2^2+[u]_s^2 \right)^{\frac{\theta}{2}} \|u\|_2^{1-\theta}=C(p)T(u)^{\frac{\theta}{2}}\| u \|_2^{1-\theta}, 
			\end{equation} 
			for all \(p\geq2\), where \( \theta=1-\frac{2}{p}. \) Then, we deduce \begin{equation}\label{Lemma3.2_F} 
				\mathcal{A}_F(u)\leq C_3\varepsilon^2\| u \|_2^{2+\alpha}T(u)+ C_4 \kappa_{\varepsilon}^2\| u \|_2^{\frac{2+\alpha}{t}}T(u)^{\frac{2qt-2-\alpha}{2t}} =
				C_3\varepsilon^2a^{2+\alpha}k + C_4\kappa_\varepsilon^2 a^{\frac{2+\alpha}{t}} k^{\frac{2qt-2-\alpha}{2t}}. 
			\end{equation} 
			Similarly, by Proposition \ref{PRc} \eqref{nona}, \eqref{nonb}, and following the arguments of Lemma \ref{geo_mixed}, we obtain
			\begin{align} 
				\int_{\mathbb{R}^2} (I_{\alpha}*F(u))f(u)u\,dx \nonumber &\leq C_5 \left( \int_{\mathbb{R}^2} |F(u)|^{\frac{4}{2+\alpha}}\,dx \right)^{\frac{2+\alpha}{4}} \left( \int_{\mathbb{R}^2} |f(u)u|^{\frac{4}{2+\alpha}}\,dx \right)^{\frac{2+\alpha}{4}} \nonumber\\ &\leq C_6\varepsilon^2a^{2+\alpha}k + C_7\kappa_\varepsilon^2 a^{\frac{2+\alpha}{t}} k^{\frac{2qt-2-\alpha}{2t}}. \label{Lemma3.2_f(u)u} \end{align} By \eqref{Lemma3.2_F} and \eqref{Lemma3.2_f(u)u}, we get 
			
			\begin{align*} 
				J(u) &= \frac{\|\nabla u\|_2^2}{2} + \frac{[u]_s^2}{2} - \frac{\Lambda}{2}  \mathcal{A}_F(u) 
				\geq \frac{k}{2} - \left( \frac{\Lambda}{2} C_3\varepsilon^2a^{2+\alpha} \right)k - \left( \frac{\Lambda}{2} C_4\kappa_\varepsilon^2 a^{\frac{2+\alpha}{t}} \right) k^{\frac{2qt-2-\alpha}{2t}},
			\end{align*}
			
			and 
			
			\begin{align*} 
				P(u) &= \|\nabla u\|_2^2 + s[u]_s^2 + \Lambda\left(\frac{2+\alpha}{2}\right) \mathcal{A}_F(u)- \Lambda \int_{\mathbb{R}^2} (I_{\alpha}*F(u))f(u)u\,dx \\ 
				&\geq s k +0 - \Lambda C_6\varepsilon^2a^{2+\alpha}k - \Lambda C_7\kappa_\varepsilon^2 a^{\frac{2+\alpha}{t}} k^{\frac{2qt-2-\alpha}{2t}}. 
			\end{align*} 
			Since \(t'\) is sufficiently close to \(1\), we may choose \(t>2\). Moreover, by choosing \(q>1\) sufficiently large, we have \( \frac{2qt-2-\alpha}{2t}>1. \) Hence we can find $K(a)>0$ such that
			$$J(u) \text{ and } P(u)>0 \text{ for all } k\in [0,K(a)].$$
			Therefore, considering $A=A_{K(a)}$, we get $J(u)>0$ and $P(u)>0$ for all $u\in A$. Further suppose, $u_1\in A$ and $u_2\in B$, then as done above in \eqref{Lemma3.2_F} we obtain 
			\begin{equation}\label{eq3.6}
				\begin{aligned}
					\mathcal{A}_F(u_2)
					&\leq
					\left(
					C_3\varepsilon^2a^{2+\alpha}
					\right)2K(a)+
					\left(
					C_4\kappa_\varepsilon^2
					a^{\frac{2+\alpha}{t}}
					\right)
					\bigl(2K(a)\bigr)^{
						\frac{2qt-2-\alpha}{2t}
					}.
				\end{aligned}
			\end{equation}
			Thus, \begin{align*} J(u_2)-J(u_1) &\geq \frac{\|\nabla u_2\|_2^2}{2} + \frac{[u_2]_s^2}{2} - \frac{\|\nabla u_1\|_2^2}{2} - \frac{[u_1]_s^2}{2}  - \frac{\Lambda}{2} \mathcal{A}_F(u_2)\\
				& = \frac{ T(u_2)}{2} - \frac{ T(u_1)}{2}  - \frac{\Lambda}{2} \mathcal{A}_F(u_2)\\
				& \geq \frac{K(a)}{2} - \left(\Lambda C_3\varepsilon^2a^{2+\alpha} \right)K(a) - \left( \frac{\Lambda}{2} C_4'\kappa_\varepsilon^2 a^{\frac{2+\alpha}{t}} \right) K(a)^{ \frac{2qt-2-\alpha}{2t}}, \end{align*} 
			where \(C_4'>0\) is the suitable constant. Since  $q$ and $t>2$, choosing $K(a)$ sufficiently small such that
			$$\left(\Lambda C_3\varepsilon^2a^{2+\alpha} \right)K(a) + \left( \frac{\Lambda}{2} C_4'\kappa_\varepsilon^2 a^{\frac{2+\alpha}{t}} \right) K(a)^{ \frac{2qt-2-\alpha}{2t}}<\frac{K(a)}{4}.$$
			Thus,  we obtain 
			\[ J(u_2)-J(u_1)\geq \frac{K(a)}{2} - \left( \Lambda C_3\varepsilon^2a^{2+\alpha} \right)K(a) - \left( \frac{\Lambda}{2} C_4'\kappa_\varepsilon^2 a^{\frac{2+\alpha}{t}} \right) K(a)^{ \frac{2qt-2-\alpha}{2t} } > \frac{K(a)}{4} > 0. \] 
			Hence, \( J(u_2)>J(u_1) \) for every $u_1\in A$ and $u_2\in B$. Therefore, 
			\[ 0<\sup_{u\in A}J(u) \leq \inf_{u\in B}J(u).\] 
		\end{proof}
		\noindent
		\noindent As a direct consequence of Lemma \ref{P1_mixed}, we have the following corollary.
		
		\begin{corollary}\label{newcor_mixed}
			Assume that \ref{f0}-\ref{f2} hold. Let $u\in \mathcal{S}_{r}(a)$, then, 
			\[
			J_*:=
			\inf
			\left\{
			J(u): u\in\mathcal S_r(a),\ T(u)= K(a)\right\}>0.
			\]
		\end{corollary}
		
		\begin{proof}
			Arguing as in the proof of Lemma \ref{P1_mixed}, we obtain 
			$$J(u)\geq \frac{K(a)}{2}-\left(\frac{\Lambda}{2}C_3\eps^2a^{2+\alpha}\right)K(a)-\left(\frac{\Lambda}{2}C_4\kappa_{\eps}^2a^{\frac{2+\alpha}{t}}\right)K(a)^{\frac{2qt-2-\alpha}{2t}}=\tilde\delta>0,$$
			for every $u\in \mathcal S_r(a)$ such that $T(u)=K(a).$
			Hence,
			\begin{eqnarray*}
				J_* & = & \inf\left\{
				J(u):
				u\in\mathcal S_r(a),\ 
				T(u) = K(a)
				\right\}
				\geq  \tilde\delta>0.
			\end{eqnarray*}
			This completes the proof.
		\end{proof}
		
		
		\noindent
		In the view of Lemmas \ref{geo_mixed} and  \ref{P1_mixed},  we are now in a position to define the minimax level.
		Following the ideas of Jeanjean~\cite{Jeanjean1997}, we define the
		class of admissible paths by
		\[
		\Gamma
		:=
		\left\{
		h\in C([0,1],\mathcal S_r(a)):
		h(0)\in A
		\ \text{and}\
		J(h(1))\leq0
		\right\},
		\]
		where \(A=A_{K(a)}\) is as defined in Lemma~\ref{P1_mixed}. The
		corresponding mountain-pass level is defined by
		\[
		\gamma(a)
		:=
		\inf_{h\in\Gamma}
		\max_{t\in[0,1]}J(h(t)).
		\]
		\begin{remark}
			$\gamma(a)>0$.
		\end{remark}
		\begin{proof}
			Let \(h\in\Gamma\) be arbitrary and
			define
			\[
			\phi_h(t):=T(h(t)),\;\; t\in[0,1].
			\]
			Since \(h\in\Gamma\), we have \(h(0)\in A\), and hence
			\[
			\phi_h(0)=T(h(0)) \leq K(a).
			\]
			Also since $J(h(1))<0$, we must have $\phi_h(1)>K(a)$, otherwise we will contradict Lemma \ref{P1_mixed}.
			Therefore, by continuity of $\phi_h$ there exists
			\(t_0\in [0,1]\) such that
			\[
			\phi_h(t_0)=T(h(t_0))=K(a).
			\]
			Consequently, by Corollary \ref{newcor_mixed} we get
			\[
			\max_{t\in[0,1]}J(h(t))
			\geq J(h(t_0))
			\geq J_*>0.
			\]
			Since \(h\in\Gamma\) was arbitrary, we conclude that
			\[
			\gamma(a)
			=
			\inf_{h\in\Gamma}\max_{t\in[0,1]}J(h(t))
			\geq J_*>0.
			\]
		\end{proof}
		\section{{Analysis of the Palais-Smale Sequence}}\label{PSa}
		{In this section, we will study the Palais-Smale sequence corresponding to $\gamma(a)$. Consider
			$$\tilde{\gamma}(a):= \inf_{\tilde{h}=(\tilde{h}_1,\tilde{h}_2)\in \tilde{\Gamma}}\max_{t\in [0,1]}J(H(\tilde{h}_1(t),\tilde{h}_2(t))),$$
			where $$\tilde{\Gamma} =\{\tilde{h}\in C([0,1], \mathcal{S}_r(a)\times \mathbb{R}): \tilde{h}(0)\in A\times \{0\} \text{ and } \tilde{h}(1)=(\tilde{h}_1(1),0) \text{ with } J(\tilde{h}_1(1))<0\}.$$
			Clearly, following the proof of \cite[Proposition~2.1]{Jeanjean1997} one can see that $\tilde{\gamma}(a)=\gamma(a)$. Hence, using the Ekeland variational principle as explained in \cite[Lemma~2.11]{Nidhi2025Normalized_Asymp}, we get sequences $\{u_n\}\subset \mathcal{S}_r(a)$ and $\{\lambda_n\}\subset \mathbb{R}$ such that as $n\rightarrow +\infty$,
			\begin{equation}\label{eq_PSS}
				\left\{\begin{array}{ll}
					J(u_n) \rightarrow \gamma(a)   & \text{ in } \mathbb{R},  \\
					J'(u_n)+\lambda_n u_n \rightarrow 0 & \text{ in } H^{-1}(\mathbb{R}^2),\\
					P(u_n) \rightarrow 0 & \text{ in }\mathbb{R}.
				\end{array}
				\right\}
			\end{equation}
			In view of the critical exponential nonlinearity, proving uniform boundedness for the (PSP)-sequence, and avoiding a bound from below such as \eqref{glb}-depends on a sharp control of the mountain pass level in Lemma \ref{MP-level}. Such control is obtained via a delicate analysis of the Moser sequence.
			\begin{lemma}\label{estimates_w_m}
				Considering the sequence of Moser functions, defined as follows:
				\begin{equation}
					M_m(x):=\frac{1}{\sqrt{2\pi}}
					\left\{
					\begin{array}{ll}
						\sqrt{\ln(m)}, & |x|\leq \frac{1}{m},\\
						\frac{\ln(1/|x|)}{\sqrt{\ln(m)}}, & \frac{1}{m}<|x|\leq1,\\
						0, & |x|>1,
					\end{array}
					\right.
				\end{equation}
				and a radially symmetric function $\eta\in C_c^{\infty}(\mathbb{R}^2)$ with $supp(\eta)\subset B_3(0)\setminus B_2(0)$ such that $\eta> 0$ and $\| \eta \|_2^2=1$; we define 
				\begin{equation}
					\eta_m(x):= \frac{1}{m}\eta\left(\frac{x}{m}\right) \text{ and } w_m(x):=M_m(x)+c_m\eta_m(x) \text{ for all } m>m_0,
				\end{equation}
				where $m_0\in \mathbb{N}$ is such that $c_m=\left(a^2- \| M_m\|_2^2\right)^{\frac{1}{2}}$ is well defined for all $m\geq m_0$. Then $w_m\in S(a)$, moreover
				$$\| \nabla w_m \|_2^2=1+o\left(\frac{1}{\ln(m)}\right) \text{ and } [w_m]_s^2\leq o\left(\frac{1}{\left(\ln(m)\right)^{\delta}}\right),$$
				for some $\delta\in (0,1-s)$.
			\end{lemma}
			\begin{proof}
				For the sequence of Moser functions, the  following result is already known;  see \cite{HuangRomani2025}
				\begin{equation}\label{moser_estimates}
					\left\{
					\begin{array}{rl}
						\| \nabla M_m\|_2^2 & = 1;\\
						\| M_m\|_2^2 & =  \frac{1}{4\ln(m)}\left(1-\left(\frac{2\ln(m)+1}{m^2}\right)\right)= \frac{1}{4\ln(m)}\left(1+o\left(\frac{1}{ln(m)}\right)\right)\rightarrow 0 \text{ as } m \rightarrow +\infty.
					\end{array}
					\right.
				\end{equation}
				Thus, we can find $m_0\in \mathbb{N}$ such that the above defined $c_m$ and $w_m$ are well defined for all $m>m_0$. Now, since $supp(\eta_m)\cap supp (M_m)\subset (B_{3m}(0)\setminus B_{2m}(0))\cap B_{1}(0)$ is an empty set, we get 
				\begin{equation*}
					\| w_m\|_2^2  =  \int_{\mathbb{R}^2}|M_m+c_m\eta_m|^2dx = \| M_m\|_2^2 +c_m^2 \| \eta_m\|_2^2+0 =a^2,
				\end{equation*}
				thus, $w_m\in S(a)$. Similarly, by \eqref{moser_estimates} we get
				\begin{eqnarray*}
					\| \nabla w_m \|_2^2 & = & \int_{\mathbb{R}^2} |\nabla M_m (x)+ \frac{c_m}{m^2}\nabla \eta\left(\frac{x}{m}\right)|^2dx = \| \nabla M_m \|_2^2 +\frac{c_m^2}{m^2}\| \nabla \eta \|_2^2+0\\
					& = & 1+\frac{(a^2-\| M_m\|_2^2)}{m^2} \| \nabla \eta \|^2 =1+o\left(\frac{1}{\ln(m)}\right).
				\end{eqnarray*}
				Next, we look at the fractional Gagliardo norm of $w_m$:
				\begin{eqnarray}\label{fractional_w_m}
					[w_m]_s^2 & = & \int_{\mathbb{R}^2}\int_{\mathbb{R}^2}\frac{|w_m(x)-w_m(y)|^2}{|x-y|^{2+2s}}dxdy = \int_{\mathbb{R}^2}\int_{\mathbb{R}^2} \frac{|\left(M_m(x)-M_m(y)\right)+c_m\left( \eta_m(x)-\eta_m(y)\right)|^2}{|x-y|^{2+2s}}dxdy\nonumber\\
					& = & [M_m]_s^2 +c_m^2 [\eta_m]_s^2 +2c_m \int_{\mathbb{R}^2}\int_{\mathbb{R}^2} \frac{\left(M_m(x)-M_m(y)\right)\left(\eta_m(x)-\eta_m(y)\right)}{|x-y|^{2+2s}}dxdy.
				\end{eqnarray}
				Now, since $\eta_m$ and $M_m$ has disjoint supports and both $\eta_m,M_m>0$, we have:
				\begin{eqnarray}\label{fractional_part}
					&& 	2c_m \int_{\mathbb{R}^2}\int_{\mathbb{R}^2} \frac{\left(M_m(x)-M_m(y)\right)\left(\eta_m(x)-\eta_m(y)\right)}{|x-y|^{2+2s}}dxdy \nonumber \\
					&& = -4c_m \int_{supp(M_m)}\int_{supp(\eta_m)}\frac{M_m(x)\eta_m(y)}{|x-y|^{2+2s}}dxdy\leq 0.
				\end{eqnarray}
				Thus, using \eqref{fractional_part}, definition of $\eta_m$ and the interpolation inequality, see \cite{Leoni2023}, ($[u]_s \leq C_s\| u \|_2^{1-s}\| \nabla u\|_2^s  \text{ for } u\in H^1(\mathbb{R}^2)) $  in \eqref{fractional_w_m} we get
				
				\begin{eqnarray*}
					[w_m]_s^2 & \leq & C_s^2 \|M_m\|_2^{2(1-s)}\| \nabla M_m \|_2^{2s}+c_m^2 \frac{[\eta]_s^2}{m^{2s}}\\
					& = & C_s^2 \left(\frac{1}{4\ln(m)}\left(1-\left(\frac{2\ln(m)+1}{m^2}\right)\right)\right)^{1-s} +\left(a^2-\| M_m\|_2^2\right)\frac{[\eta]_s^2}{m^{2s}}\\
					& \leq & \frac{C_s}{4^{1-s}\left(\ln(m)\right)^{1-s}} + \left(a^2-\| M_m\|_2^2\right)\frac{[\eta]_s^2}{m^{2s}} = o\left(\frac{1}{\left(\ln(m)\right)^{\delta}}\right).
				\end{eqnarray*}
			\end{proof}
			\noindent {
				For every sufficiently large \(m\), define \(g_m:\mathbb R^+\to\mathbb R\) by
				\begin{equation}\label{gm}
					g_m(t)
					:=
					J(H(w_m,\ln t))
					=
					\frac{t^2}{2}\|\nabla w_m\|_2^2
					+
					\frac{t^{2s}}{2}[w_m]_s^2
					-
					\frac{\Lambda}{2t^{2+\alpha}}
					\mathcal A_F(tw_m).
				\end{equation}
				We assert that there exists \(t_m>0\) such that
				\[
				g_m(t_m)=\max_{t>0}g_m(t).
				\]
				Indeed, following the argument in Lemma~\ref{geo_mixed}, one has
				\[
				g_m(t)\to-\infty \text{ as }t\to+\infty; \text{ and }
				g_m(t)\to0 \text{ as }t\to0^+.
				\]
				Moreover, since \(w_m\in S_r(a)\), there exists \(\bar t_m>0\)
				such that
				\[
				T(H(w_m,\ln\bar t_m))
				=
				\bar t_m^2\|\nabla w_m\|_2^2
				+
				\bar t_m^{2s}[w_m]_s^2
				=
				K(a).
				\]
				Hence, by Corollary~\ref{newcor_mixed},
				\[
				g_m(\bar t_m)
				=
				J(H(w_m,\ln\bar t_m))
				\ge J_*>0.
				\]
				Therefore, continuity of \(g_m\) ensures the claim. We now turn to deriving a suitable upper bound for \(g_m(t_m).\)
			}

			\begin{lemma}\label{MP-level}
				Assume that \ref{f0}-\ref{f3} and \ref{f5} hold. Then
				\[
				\limsup_{m\to+\infty} g_m(t_m)
				\leq
				\beta
				:=
				\frac{(2+\alpha)\pi}{2\gamma_0},
				\]
				where \(t_m>0\) is such that
				\(
				g_m(t_m)=\displaystyle\max_{t>0}g_m(t).
				\)
			\end{lemma}
			\begin{proof}
				
				We first observe that the sequence \(\{t_m\}\) stays away from
				zero. Indeed, since \(t_m\) is the maximum point of \(g_m\),
				\[
				g_m(t_m)\geq g_m(\bar t_m)\geq J_*>0.
				\]
				On the other hand, since \(\mathcal A_F(t_m w_m)\geq0\), from
				\eqref{gm} we have
				\[
				g_m(t_m)
				\leq
				\frac{t_m^2}{2}\| \nabla w_m \|_2^2
				+
				\frac{t_m^{2s}}{2}[w_m]_s^2.
				\]
				If, along a subsequence, \(t_m\to0\), then Lemma \ref{estimates_w_m} 
				would imply
				\(
				g_m(t_m)\to0,
				\)
				which contradicts the fact that \(g_m(t_m)\geq J_*>0\). Hence
				\begin{equation}\label{tm-positive}
					\liminf_{m\to+\infty}t_m>0.
				\end{equation}
				Further since $t_m$ is a critical point of $g_m$, by \ref{f3} we get
				\begin{eqnarray*}
					&& t_m^2\| \nabla w_m \|_2^2 +st_m^{2s} [w_m]_s^2 \\
					&& = t_m g'_m(t_m) +\frac{\Lambda}{t_m^{2+\alpha}} \int_{\mathbb{R}^2}(I_{\alpha}*F(t_mw_m))\left(f(t_mw_m)t_m w_m -\left(\frac{2+\alpha}{2}\right)F(t_mw_m)\right)\\
					& & \geq 0+\frac{\Lambda}{t_m^{2+\alpha}}\left( \mu -\left(\frac{2+\alpha}{2}\right)\right)\mathcal{A}_{F}(t_mw_m) \\
					&& \geq \frac{\Lambda}{t_m^{2+\alpha}}\left(\mu - \left(\frac{2+\alpha}{2}\right)\right) \int_{B_{\frac{1}{m}}(0)}\int_{B_{\frac{1}{m}}(0)} \frac{A_{\alpha}F(A_m)^2}{|x-y|^{2-\alpha}}dxdy\\
					&& \geq \frac{\Lambda}{t_m^{2+\alpha}}\left(\mu - \left(\frac{2+\alpha}{2}\right)\right) A_{\alpha}\left(\frac{m}{2}\right)^{2-\alpha}\left(\frac{\pi}{m^2}\right)^2\frac{\left(F(A_m)A_m^2\right)^2}{A_m^4}
				\end{eqnarray*}
				where $A_m = t_mw_m(x)= t_m\sqrt{\frac{\ln(m)}{2\pi}}$ for all $x\in B_{\frac{1}{m}}(0)$. Clearly, $A_m\rightarrow +\infty\text{ as } m\rightarrow +\infty$. Further, using the growth condition \eqref{lower-F-critical}, we get:
				\begin{equation}\label{eq:key-estimate}
					t_m^2\|\nabla w_m \|_2^2 +st_m^{2s} [w_m]_s^2 \geq  \frac{\Lambda}{t_m^{2+\alpha}}\left(\mu -\left(\frac{2+\alpha}{2}\right)\right)\frac{A_{\alpha}\pi^4\beta_0^2e^{\frac{2\gamma_0t_m^2}{2\pi}\ln(m)}}{2^{2-\alpha}m^{2+\alpha}\gamma_0^2t_m^4\left(\ln(m)\right)^2} = C_0 \frac{m^{\frac{\gamma_0t_m^2}{\pi}-(2+\alpha)}}{t_m^{6+\alpha}\left(\ln(m)\right)^{2}}.
				\end{equation} 
				We claim that \(\{t_m\}\) is bounded.
				Indeed, suppose by contradiction that, along a subsequence,
				\(t_m\to+\infty\). Then, the left-hand side of
				\eqref{eq:key-estimate} has at most polynomial growth in
				\(t_m\), whereas the right-hand side contains the factor
				\[
				m^{\frac{\gamma_0}{\pi}t_m^2-(2+\alpha)}
				=
				e^{
					\left(
					\frac{\gamma_0}{\pi}t_m^2-(2+\alpha)
					\right)\ln m},
				\]
				which dominates every polynomial factor in \(t_m\) and
				\(\ln m\). Hence we get a contradiction. Thus \(\{t_m\}\) is bounded. Moreover, 
				\[
				\limsup_{m\to+\infty}t_m^2
				\leq
				\frac{(2+\alpha)\pi}{\gamma_0},
				\]
				else, we will end up with similar contradiction. Therefore, by Lemma \ref{estimates_w_m}, we have
				\begin{eqnarray*}
					g_m(t_m) & = & \frac{t_m^2}{2} \| \nabla w_m \|_2^2 +\frac{t_m^{2s}}{2}[w_m]_s^2 -\frac{\Lambda}{2 t_m^{2+\alpha}} \mathcal{A}_F(t_mw_m)
					\leq  \frac{t_m^2}{2} \left(1+o\left(\frac{1}{\ln(m)}\right)\right) +\frac{t_m^{2s}}{2}o\left(\frac{1}{\ln(m)^{\delta}}\right)\\
					& \leq & \frac{\pi(2+\alpha)}{2\gamma_0} \left(1+o\left(\frac{1}{\ln(m)}\right)\right)+ \frac{\pi^s(2+\alpha)^s}{2\gamma_0^s}o\left(\frac{1}{\ln(m)^{\delta}}\right)
					\rightarrow \frac{\pi (2+\alpha)}{2\gamma_0} \text{ as } m\rightarrow +\infty.
				\end{eqnarray*}
				Hence 
				\[ \limsup_{m\to+\infty} \max_{t>0}g_m(t) \leq \beta. \]
			\end{proof}	
			\begin{corollary}\label{upperbound_mixed}
				Assume that \ref{f0}-\ref{f3} and \ref{f5} hold. Then
				\[
				\gamma(a)
				\leq \beta=
				\frac{(2+\alpha)\pi}{2\gamma_0}.
				\]
			\end{corollary}
			\begin{proof}
				For every sufficiently large \(m\), by Lemma~\ref{geo_mixed}, we have
				\[
				T(H(w_m,\sigma))
				=
				\|\nabla H(w_m,\sigma)\|_2^2
				+
				[H(w_m,\sigma)]_s^2
				\longrightarrow
				\begin{cases}
					0, & \text{as } \sigma\to-\infty,\\
					+\infty, & \text{as } \sigma\to+\infty.
				\end{cases}
				\]
				Hence, there exists \(\sigma_{1,m}<0\), sufficiently negative, such that
				\(
				T(H(w_m,\sigma_{1,m}))\leq K(a),
				\)
				and therefore
				\(
				H(w_m,\sigma_{1,m})\in A.
				\)
				Moreover, by Lemma~\ref{geo_mixed},
				\[
				J(H(w_m,\sigma))\longrightarrow-\infty
				\qquad\text{as }\sigma\to+\infty.
				\]
				Thus, we may choose \(\sigma_{2,m}>0\), sufficiently large, such that
				\(
				J(H(w_m,\sigma_{2,m}))<0.
				\)
				Now define
				\[
				h_m(\tau)
				:=
				H\bigl(w_m,(1-\tau)\sigma_{1,m}+\tau\sigma_{2,m}\bigr),
				\qquad \tau\in[0,1].
				\]
				Since the scaling \(H\) preserves the \(L^2\)-norm,
				\[
				h_m(\tau)\in \mathcal S_r(a)
				\qquad\text{for every }\tau\in[0,1].
				\]
				Moreover,
				\[
				h_m(0)=H(w_m,\sigma_{1,m})\in A,
				\qquad
				J(h_m(1))=J(H(w_m,\sigma_{2,m}))<0.
				\]
				Hence \(h_m\in\Gamma\). Therefore, by the definition of the
				mountain-pass level,
				\[
				\begin{aligned}
					\gamma(a)
					\leq
					\max_{\tau\in[0,1]}J(h_m(\tau))\leq
					\sup_{\sigma\in\mathbb R}J(H(w_m,\sigma))=
					\max_{t>0}g_m(t).
				\end{aligned}
				\]
				Taking the \(\limsup\) as \(m\to+\infty\) and using
				Lemma~\ref{MP-level}, we obtain
				\[
				\gamma(a)
				\leq
				\limsup_{m\to+\infty}\max_{t>0}g_m(t)
				\leq
				\beta
				=
				\frac{(2+\alpha)\pi}{2\gamma_0}.
				\]
				This completes the proof.
			\end{proof}

			\noindent We now proceed to verify that the $(PSP)_{\gamma(a)}$-conditions hold for the constrained functional $J|_{\mathcal{S}_r(a)}$. To this aim, we begin by proving a few convergence results.
			\begin{lemma}\label{TMineq}
				Let $\{u_n\}\subset S_r(a)$ be a sequence satisfying
				\[
				\limsup_{n\to+\infty}T(u_n)
				<
				2\beta
				=
				\frac{(2+\alpha)\pi}{\gamma_0}.
				\]
				Then there exist $\gamma>\gamma_0$, sufficiently close to $\gamma_0$,
				$t>1$, sufficiently close to $1$, and a constant $C>0$ such that
				\[
				\int_{\mathbb R^2}
				\bigl(e^{\gamma|u_n|^2}-1\bigr)^t\,dx
				\leq C,
				\qquad \forall\, n\in\mathbb N.
				\]
			\end{lemma}
			
			\begin{proof}
				Since
				\(\displaystyle
				\limsup_{n\to+\infty}T(u_n)<2\beta,
				\)
				there exist $\varepsilon>0$ and $n_0\in\mathbb N$ such that
				\[
				T(u_n)\leq 2\beta-\varepsilon,
				\qquad \forall\, n\geq n_0.
				\]
				Set
				\(
				K:=2\beta-\varepsilon>0.
				\)
				Since
				\(
				2\beta\gamma_0=(2+\alpha)\pi<4\pi,
				\)
				we may choose $\gamma>\gamma_0$, sufficiently close to $\gamma_0$,
				and then $t>1$, sufficiently close to $1$, such that
				\(
				t\gamma K<4\pi.
				\)
				For $n\geq n_0$, define
				\(
				v_n:=\frac{u_n}{\sqrt K}.
				\)
				Then
				\[
				\|\nabla v_n\|_2^2
				=
				\frac{\|\nabla u_n\|_2^2}{K}
				\leq
				\frac{T(u_n)}{K}
				\leq1,
				\]
				and, since $u_n\in S_r(a)$,
				\(
				\|v_n\|_2^2=\frac{a^2}{K}.
				\)
				Therefore, by Proposition~\ref{PRa} with
				\(
				\theta:=t\gamma K<4\pi,
				\)
				there exists a constant $C_1>0$, independent of $n$, such that
				\[
				\int_{\mathbb R^2}
				\left(e^{\theta|v_n|^2}-1\right)\,dx
				\leq C_1,
				\qquad n\geq n_0.
				\]
				Since
				\(
				\theta|v_n|^2
				=
				t\gamma|u_n|^2,
				\)
				and
				\[
				(e^\eta-1)^t
				\leq C_t\bigl(e^{t\eta}-1\bigr),
				\qquad \eta\geq0,
				\]
				we obtain
				\[
				\begin{aligned}
					\int_{\mathbb R^2}
					\bigl(e^{\gamma|u_n|^2}-1\bigr)^t\,dx
					\leq
					C_t
					\int_{\mathbb R^2}
					\left(e^{t\gamma|u_n|^2}-1\right)\,dx=
					C_t
					\int_{\mathbb R^2}
					\left(e^{\theta|v_n|^2}-1\right)\,dx\leq C_tC_1,
				\end{aligned}
				\]
				for every $n\geq n_0$.
				
				Finally, by Proposition~\ref{PRa}(i), each of the finitely many
				integrals corresponding to $n<n_0$ is finite. Hence, enlarging the
				constant if necessary, we obtain
				\[
				\int_{\mathbb R^2}
				\bigl(e^{\gamma|u_n|^2}-1\bigr)^t\,dx
				\leq C,
				\qquad \forall\, n\in\mathbb N.
				\]
				This completes the proof.
			\end{proof}

			\begin{lemma}\label{imp} Assume that \ref{f0}-\ref{f3} and \ref{f5} hold. Let $\{u_{n}\}\subset \mathcal{S}_{r}(a)$ with 
				{\begin{equation}\label{ene}
						\begin{aligned}\displaystyle
							\limsup_{n \rightarrow +\infty}  T(u_n)  < 2\beta=\frac{(2+\alpha)\pi}{\gamma_0}.
						\end{aligned}
				\end{equation}}
				If $u_{n} \rightharpoonup u$ in $H^{1}_{rad}(\mathbb{R}^{2})$ and $u_n(x) \rightarrow u(x)$ a.e. in $\mathbb{R}^2$, then
				\begin{eqnarray*}
					\begin{aligned}\displaystyle
						|u_n|^{q}(e^{\gamma  |u_n|^{2}}-1) \rightarrow |u|^{q}(e^{\gamma |u|^{2}}-1) \,\, \mbox{in} \,\, L^{1}(\mathbb{R}^{2}).
					\end{aligned}
				\end{eqnarray*}
			\end{lemma}

			\begin{proof}
				Setting $\bar{h}_n(x)=e^{\gamma |u_n(x)|^{2}}-1$.
				By Lemma \ref{TMineq}, we know that
				\begin{eqnarray*}
					\begin{aligned}\displaystyle
						\int_{\mathbb{R}^{2}}\Big(e^{\gamma |u_{n}|^{2}}-1\Big)^{t}dx&
						\leq \int_{\mathbb{R}^{2}}\Big(e^{t\gamma |u_{n}|^{2}}-1\Big)dx<C, \text{ for all } n\in \mathbb{N}					\end{aligned}
				\end{eqnarray*}
				where $C=C(t,a,\gamma)>0$. Then,
				\begin{eqnarray*}
					\begin{aligned}\displaystyle
						\int_{\mathbb{R}^{2}} \Big(\bar{h}_n(x)\Big)^{t}dx=\int_{\mathbb{R}^{2}}\Big(e^{\gamma |u_{n}(x)|^{2}}-1\Big)^{t}dx<C,
					\end{aligned}
				\end{eqnarray*}
				which implies that
				$\bar{h}_{n} \in L^{t}(\mathbb{R}^{2}) \mbox{ and } \displaystyle \sup_{n \in \mathbb{N}}\|\bar{h}_n\|_{t}<+\infty.$
				Therefore, $\{\bar{h}_{n}\}$ is a bounded sequence in $L^{t}(\mathbb{R}^{2})$. 
				Now, since \(u_n \rightharpoonup u\) in \(H^{1}_{rad}(\mathbb{R}^{2})\) and
				\(u_n \rightarrow u\) a.e. in \(\mathbb{R}^{2}\), by
				\cite[Lemma 4.8]{Kavian1993}, we obtain
				\begin{equation}\label{sNewlimit}
					\bar{h}_n \rightharpoonup
					\bar{h}:=e^{\gamma |u|^{2}}-1
					\mbox{ in } L^{t}(\mathbb{R}^{2}).
				\end{equation}
				Further, by the compact embedding
				\(H^{1}_{rad}(\mathbb{R}^{2})\hookrightarrow L^{qt'}(\mathbb{R}^{2})\),
				we have
				\[
				u_n\rightarrow u
				\mbox{ in }L^{qt'}(\mathbb{R}^{2}), \text{ where }
				t'=\frac{t}{t-1}.
				\]
				Therefore,
				\begin{equation}\label{ssubimp}
					|u_n|^{q}\rightarrow |u|^{q}
					\mbox{ in }L^{t'}(\mathbb{R}^{2}).
				\end{equation}
				Moreover, \(\{\bar h_n\}\) is bounded in \(L^t(\mathbb{R}^2)\) and
				\(\bar h_n\to\bar h\) a.e. in \(\mathbb{R}^2\). Hence, by
				H\"older's inequality, Vitali's convergence theorem on bounded sets,
				and a standard truncation argument, \eqref{sNewlimit} and
				\eqref{ssubimp} yield
				\[
				|u_n|^{q}\bigl(e^{\gamma|u_n|^{2}}-1\bigr)
				\rightarrow
				|u|^{q}\bigl(e^{\gamma|u|^{2}}-1\bigr)
				\mbox{ in }L^{1}(\mathbb{R}^{2}).
				\]
				This completes the proof.

			\end{proof}
			\noindent	By Lemma \ref{imp}, we have the following two important corollaries.
			\begin{corollary}\label{weaklim}
				Assume that \ref{f0}-\ref{f3} and \ref{f5} hold. Let $\{u_{n}\}\subset \mathcal{S}_{r}(a)$ with
				\begin{equation}
					\begin{aligned}\displaystyle
						\limsup_{n \rightarrow +\infty}  T(u_n)  < 2\beta = \frac{(2+\alpha)\pi}{\gamma_0}.
					\end{aligned}
				\end{equation}
				If $u_{n} \rightharpoonup u$ in $H^{1}_{rad}(\mathbb{R}^{2})$ and $u_{n}(x) \rightarrow u(x)$ a.e in $\mathbb{R}^2$, then
				\begin{eqnarray*}
					\begin{aligned}\displaystyle
						\int_{\mathbb{R}^{2}}(I_{\alpha}\ast F(u_{n}))f(u_{n})\psi dx \rightarrow \int_{\mathbb{R}^{2}}(I_{\alpha}\ast F(u))f(u)\psi dx, \ \mbox{as} \ n\rightarrow\infty,
					\end{aligned}
				\end{eqnarray*}
				for any $\psi\in C_{0}^{\infty}(\mathbb{R}^{2})$.
			\end{corollary}
			
			\begin{proof}
				As we know from, \cite[Lemma 4.1]{ACTT},
				\begin{equation}\label{choq}
					\begin{aligned}\displaystyle
						\|I_{\alpha}\ast F(u_{n})\|_{\infty}\leq \bar{C}.
					\end{aligned}
				\end{equation}
				Hence, for any $\psi\in C_{0}^{\infty}(\mathbb{R}^{2})$, we have
				\begin{eqnarray*}
					\begin{aligned}\displaystyle
						|(I_{\alpha}\ast F(u_{n}))f(u_{n})\psi|\leq \bar{C}|f(u_{n})||\psi|\leq
						\bar{C}\varepsilon |u_n|^{1+\frac{\alpha}{2}} |\psi|
						+
						\bar{C}\kappa_{\varepsilon}|u_n|^{q-1} |\psi|
						\left(e^{\gamma |u_n|^2}-1\right).
					\end{aligned}
				\end{eqnarray*}
				Let $U=\text{supp }(\psi)$. Then, we obtain
				\begin{eqnarray*}
					\begin{aligned}\displaystyle
						\int_{U}|u_{n}|^{{1+\frac{\alpha}{2}}}|\psi|dx\rightarrow\int_{U}|u|^{{1+\frac{\alpha}{2}}}|\psi|dx, \ \mbox{as} \ n\rightarrow\infty,
					\end{aligned}
				\end{eqnarray*}
				and
				\begin{eqnarray*}
					\begin{aligned}\displaystyle
						\int_{U}|u_{n}|^{q-1}|\psi|(e^{\gamma |u_{n}|^{2}}-1)dx\rightarrow\int_{U}|u|^{q-1}|\psi|(e^{\gamma |u|^{2}}-1)dx, \ \mbox{as} \ n\rightarrow\infty.
					\end{aligned}
				\end{eqnarray*}
				By the growth estimates
				\eqref{nonb}-\eqref{nona}, the compact radial embedding,  we obtain
				\[
				F(u_n)\to F(u)
				\text{ in }L^{\frac{4}{2+\alpha}}(\mathbb{R}^2),
				\]
				and
				\[
				f(u_n)\psi\to f(u)\psi
				\text{ in }L^{\frac{4}{2+\alpha}}(\mathbb{R}^2).
				\]
				Hence, by the Hardy-Littlewood-Sobolev inequality(Proposition \ref{PRc}),
				\[
				\|I_\alpha*(F(u_n)-F(u))\|_{\frac{4}{2-\alpha}}
				\leq
				C\|F(u_n)-F(u)\|_{\frac{4}{2+\alpha}}
				\longrightarrow0.
				\]
				Therefore,
				\[
				\begin{aligned}
					&\left|
					\int_{\mathbb{R}^2}
					\bigl[(I_\alpha*F(u_n))f(u_n)
					-(I_\alpha*F(u))f(u)\bigr]\psi\,dx
					\right|\\
					&\leq
					\|I_\alpha*(F(u_n)-F(u))\|_{\frac{4}{2-\alpha}}
					\|f(u_n)\psi\|_{\frac{4}{2+\alpha}}\\
					&\quad+
					\|I_\alpha*F(u)\|_{\frac{4}{2-\alpha}}
					\|(f(u_n)-f(u))\psi\|_{\frac{4}{2+\alpha}}
					\longrightarrow0.
				\end{aligned}
				\]
				Thus,
				\[
				\int_{\mathbb{R}^{2}}(I_{\alpha}\ast F(u_{n}))f(u_{n})\psi\,dx
				\longrightarrow
				\int_{\mathbb{R}^{2}}(I_{\alpha}\ast F(u))f(u)\psi\,dx.
				\]
				This completes the proof.
			\end{proof}
			
			\begin{corollary} \label{Conver} Assume that \ref{f0}-\ref{f3} and \ref{f5} hold. Let $\{u_{n}\}\subset \mathcal{S}_{r}(a)$ with $\displaystyle \limsup_{n\rightarrow +\infty} T(u_n)<2\beta$.
				If $u_{n} \rightharpoonup u$ in $H^{1}_{rad}(\mathbb{R}^{2})$ and $u_{n}(x) \rightarrow u(x)$ a.e in $\mathbb{R}^2$, then $\displaystyle \lim_{n\rightarrow+\infty} \mathcal{A}_F(u_n)=\mathcal{A}_F(u)$ and
				\begin{eqnarray*}
					\begin{aligned}\displaystyle
						\lim_{n\to+\infty}\int_{\mathbb{R}^{2}}(I_{\alpha}\ast F(u_{n}))f(u_{n})u_{n}\,dx \rightarrow \int_{\mathbb{R}^{2}}(I_{\alpha}\ast F(u))f(u)u\,dx.
					\end{aligned}
				\end{eqnarray*}	
			\end{corollary}
			\begin{proof} From \eqref{choq}, we have
				$\displaystyle	\|I_{\alpha}\ast F(u_{n})\|_{\infty}\leq \bar{C}.$
				By 
				\eqref{nona}, we have
				\begin{eqnarray*}
					\begin{aligned}\displaystyle
						|F(u_{n})|\leq
						\varepsilon |u_n|^{2+\frac{\alpha}{2}} 
						+
						C |u_n|^{q}
						\left(e^{\gamma |u_n|^2}-1\right)\, \text{ for all }\, u_{n} \in H^{1}(\mathbb{R}^{2}),
					\end{aligned}
				\end{eqnarray*}
				where $\gamma>\gamma_{0}$, $2+\frac{\alpha}{2}>2$ and $q>1$. Hence, we have
				\begin{eqnarray*}
					\begin{aligned}\displaystyle
						|(I_{\alpha}\ast F(u_{n}))F(u_{n})|\leq C|F(u_{n})|\leq C
						\varepsilon |u_n|^{2+\frac{\alpha}{2}}
						+C |u_n|^{q} 
						\left(e^{\gamma |u_n|^2}-1\right).
					\end{aligned}
				\end{eqnarray*}
				By 	Lemma \ref{imp}, we know
				\begin{eqnarray*}
					\begin{aligned}\displaystyle
						\int_{\mathbb{R}^{2}}|u_n|^{q}(e^{\gamma  |u_n(x)|^{2}}-1)dx \rightarrow \int_{\mathbb{R}^{2}} |u|^{q}(e^{\gamma |u(x)|^{2}}-1)dx \,\, \mbox{as} \,\, n\rightarrow\infty.
					\end{aligned}
				\end{eqnarray*}
				By the compact embedding $H^{1}_{rad}(\mathbb{R}^{2}) \hookrightarrow L^{p}(\mathbb{R}^{2})$, for $p>2$, we have
				\begin{eqnarray*}
					\begin{aligned}\displaystyle
						u_{n}\rightarrow u \mbox{ in } L^{p}(\mathbb{R}^2).
					\end{aligned}
				\end{eqnarray*}
				A similar argument works to show that
				\begin{eqnarray*}
					\begin{aligned}\displaystyle
						\int_{\mathbb{R}^{2}}(I_{\alpha}\ast F(u_{n}))f(u_{n})u_{n}\,dx \rightarrow \int_{\mathbb{R}^{2}}(I_{\alpha}\ast F(u))f(u)u\,dx \,\, \mbox{as} \,\, n\rightarrow\infty,
					\end{aligned}
				\end{eqnarray*}	
				which completes the proof.
			\end{proof}
			\noindent We are now ready to establish the \((PSP)_{\gamma(a)}\)- condition.
			\begin{lemma}\label{uniform-bound}
				Let
				$\{u_n\}\subset \mathcal S_r(a)$ be a Palais-Smale sequence at the
				mountain-pass level $\gamma(a)$ and $\{\lambda_n\}\subset \mathbb{R}$ be a sequence of Lagrange multipliers, satisfying the following:
				\begin{equation*}
					\left\{
					\begin{array}{rll}
						J(u_n) & \rightarrow \gamma(a)<\beta &\text{ in } \mathbb{R}; \\
						J'(u_n)+\lambda_n u_n& \rightarrow 0 &\text{ in } H^{-1}(\mathbb{R}^2) ;\\
						P(u_n) & \rightarrow 0 &\text{ in } \mathbb{R},
					\end{array}\right.
				\end{equation*}
				as $n\to \infty$. Assuming \ref{f0}-\ref{f5} holds.
				Then \(\{u_n\}\) is bounded with
				\begin{equation}\label{eqgamma}
					\limsup_{n\to\infty} T(u_n)\le
					\left(\frac{2\left(2\mu-\left(2+\alpha\right)\right)}
					{2\mu-(4+\alpha)}\right)\gamma(a).
				\end{equation}
				{Further, there exists $0\neq u_a \in H^1(\mathbb{R}^2)$ and $0\neq\lambda_a\in \mathbb{R}$ such that $u_n \rightarrow u_a$ strongly in $H^1(\mathbb{R}^2)$ and $\left( u_a, \lambda_a\right)$ solves \eqref{aa}.}
			\end{lemma}
			\begin{proof}
				Since $P(u_n)=o(1)$, by \ref{f3} we get
				\begin{eqnarray*}
					T(u_n) & \geq & \| \nabla u_n \|_2^2 +s[u_n]_s^2 = \Lambda \int_{\mathbb{R}^2}(I_{\alpha}*F(u_n))f(u_n)u_n -\Lambda\left(\frac{2+\alpha}{2}\right)\mathcal{A}_F(u_n)+o(1)\\
					& \geq &  \Lambda \left(\mu - \left(\frac{2+\alpha}{2}\right)\right)\mathcal{A}_F(u_n) +o(1),
				\end{eqnarray*}
				thus $$\Lambda\mathcal{A}_F(u_n)\leq \frac{2T(u_n)}{\left(2\mu-\left(2+\alpha\right)\right)}+o(1).$$
				Hence,
				\begin{eqnarray*}
					\gamma(a) & = & J(u_n) +o(1) = \frac{T(u_n)}{2}-\frac{\Lambda}{2}\mathcal{A}_F(u_n) +o(1)
					\geq  \frac{1}{2}\left(1-\frac{2}{2\mu-2-\alpha}\right)T(u_n) +o(1),
				\end{eqnarray*}
				therefore,
				$$\limsup_{n\to+\infty}T(u_n)\leq \left(\frac{2\left(2\mu -2-\alpha\right)}{2\mu -4 -\alpha}\right)\gamma(a)$$
				Moreover, since, \(\mu>2+\frac{\alpha}{2}\), the sequence $\{T(u_n)\}$ is bounded. It follows that the sequence \(\{\|u_n\|\}\) is bounded as well, since \(\{u_n\}\in \mathcal{S}_r(a)\). Next, we claim that the sequences $\{\mathcal{A}_F(u_n)\}, \{\lambda_n\}$ and $\{\bar\eta(u_n)\}$, where
				\[
				\bar{\eta}(u_n):=\int_{\mathbb R^2}(I_\alpha*F(u_n))f(u_n) u_n\,dx,
				\]
				are bounded.
				
				\noindent Indeed, the boundedness of $\{\mathcal{A}_F(u_n)\}$ follows from \eqref{eqgamma} together with the convergence
				\(J(u_n)\to\gamma(a)\). Further, the boundedness of $\{\bar{\eta}(u_n)\}$ is a direct consequence of the fact that the sequence $\{\mathcal{A}_F(u_n)\}$ is bounded, \(P(u_n)\to0\) and  the boundedness of
				\(\{\|u_n\|\}\). Finally, the boundedness of the Lagrange multipliers follows from
				\begin{equation}\label{lambda-mixed-formula}
					\lambda_n
					=
					\left(\frac{s-1}{a^2}\right)[u_n]_s^2
					+
					\Lambda\left(\frac{2+\alpha}{2a^2}\right)
					\int_{\mathbb R^2}(I_\alpha*F(u_n))f(u_n)u_ndx +o(1),
				\end{equation}
				which is obtained by \(P(u_n)\to 0\) and $J'(u_n)(u_n)+\lambda_n\| u_n \|_2^2=o(1)$. {
					Therefore, \(\{\lambda_n\}\) is bounded. Hence, up to a subsequence,
					\(\lambda_n\to\lambda_a
					~~\text{for some }\lambda_a\in\mathbb R.
					\)
				}
				{ Moreover, we will see that $\lambda_a \neq 0$.\\
					Claim 1: $\lambda_a\neq0$.\\
					Suppose, if $\lambda_a=0$, then by \eqref{lambda-mixed-formula} and the fact that $$0= \lim_{n\rightarrow \infty}\left(J'(u_n)(u_n)+\lambda_n \| u_n\|_2^2\right)=\lim_{n\rightarrow \infty} J'(u_n)(u_n),$$
					we get
					\begin{eqnarray*}
						0 & = & \lim_{n\rightarrow \infty} \left( J'(u_n)(u_n) +\left(\frac{2a^2}{(2+\alpha)} \right)\lambda_n \right) \\
						& = & \lim_{n\rightarrow \infty}\left( \| \nabla u_n \|_2^2 +[u_n]_s^2 -\Lambda \int_{\mathbb{R}^2}(I_{\alpha}*F(u_n))f(u_n)u_n+\left(\frac{2a^2}{(2+\alpha)}\right)\left(\left(\frac{s-1}{a^2}\right)[u_n]_s^2\right. \right.\\
						&& \;\;\;\;\;\;\;\left. \left. +	\Lambda\left(\frac{2+\alpha}{2a^2}\right)
						\int_{\mathbb R^2}(I_\alpha*F(u_n))f(u_n)u_n,dx \right)\right)\\
						& = & \lim_{n\rightarrow \infty} \left(\| \nabla u_n \|_2^2 +\left(\frac{\alpha+2s}{\alpha +2}\right)[u_n]_s^2\right) \geq \left(\frac{\alpha +2s }{\alpha +2}\right) \lim_{n\rightarrow \infty}T(u_n)\geq 0.
					\end{eqnarray*}
					Thus $\displaystyle \lim_{n\rightarrow \infty} T(u_n)=0$. But this will imply that,
					\begin{eqnarray}
						0  < \gamma (a) = \lim_{n\rightarrow \infty} J(u_n) = \lim_{n\rightarrow \infty} \left(\frac{T(u_n)}{2} -\Lambda \frac{\mathcal{A}_F(u_n)}{2}\right)  \leq 0.
					\end{eqnarray}
					Therefore, we get a contradiction and hence $\lambda_a\neq 0$.
				}
				\noindent Now, since $\{u_n\}$ is bounded in $H^1_{r}(\mathbb{R}^2)$, up to a subsequence,
				\(
				u_n\rightharpoonup u_a
				\text{ weakly in }H_{rad}^1(\mathbb R^2),
				\)
				for some \(u_a\in H_{rad}^1(\mathbb R^2)\).
				
				\noindent Claim 2: 
				\(
				u_a\not\equiv0.
				\)
				
				\noindent Assume, by contradiction, that
				\(
				u_a\equiv0.
				\)
				Since
				\(
				u_n\rightharpoonup0
				~~\text{weakly in }H_{rad}^1(\mathbb R^2),
				\)
				the compact radial embedding
				\(
				H_{rad}^1(\mathbb R^2)
				\hookrightarrow
				L^p(\mathbb R^2),
				~~
				p>2,
				\)
				implies
				\(
				u_n \to 0
				~~\text{strongly in }
				L^p(\mathbb R^2),
				\)
				for every \(p>2\). In particular,
				\[
				\|u_n\|_{2+\frac{\alpha}{2}}
				=o(1).
				\]
				Using \eqref{F-f-estimate}
				we infer that
				\begin{align}
					\mathcal A_F(u_n)
					&=
					\int_{\mathbb R^2}
					(I_\alpha*F(u_n))
					F(u_n)\,dx\leq \varepsilon \int_{\mathbb{R}^2}(I_{\alpha}*F(u_n))f(u_n)u_n +\kappa_{\varepsilon}\| (I_{\alpha}*F(u_n))\|_{\infty}\| u_n\|_{2+\frac{\alpha}{2}}^{2+\frac{\alpha}{2}}
					\notag\\
					&\le
					\varepsilon
					\int_{\mathbb R^2}
					(I_\alpha*F(u_n))
					f(u_n)u_n\,dx
					+
					C_\varepsilon
					\|u_n\|_{2+\frac{\alpha}{2}}^{\,2+\frac{\alpha}{2}}
					\notag
					=
					\varepsilon
					\int_{\mathbb R^2}
					(I_\alpha*F(u_n))
					f(u_n)u_n\,dx
					+
					o(1).
					\label{AF-small}
				\end{align}
				Since the sequence \(\{\bar{\eta}(u_n)\}\) is bounded, by the arbitrariness of \(\varepsilon\), we conclude that
				\(\displaystyle \lim_{n\to \infty}\mathcal{A}_F(u_n)=0\) and \(J(u_n)\to\gamma(a)\), we obtain
				Hence 
				\begin{equation*}
					\gamma(a)
					=
					\frac12\lim_{n\to\infty}
					\left(
					\|\nabla u_n\|_2^2+[u_n]_s^2
					\right)
					=
					\frac12\lim_{n\to\infty}T(u_n).\end{equation*}
				Equivalently,
				\(
				\displaystyle \lim_{n\to\infty}T(u_n)=2\gamma(a)\leq 2\beta.
				\)
				Then, by Corollary \ref{Conver},
				\[\lim_{n\to\infty} \int_{\mathbb R^2}
				(I_\alpha*F(u_n))
				f(u_n)u_n\,dx=0\]
				Combining the above with \(P(u_n)\to 0\), as \(n\to \infty\), we deduce that \(\|\nabla u_n\|_2^2+s[u_n]_s^2\to 0\)
				, which implies that $T(u_n)\rightarrow 0$ and hence $$\gamma(a)=\displaystyle \lim_{n\rightarrow \infty} \left(\frac{T(u_n)}{2}-\frac{\Lambda}{2}\mathcal{A}_F(u_n)\right)=0.$$
				This is a contradiction to the fact that \(\gamma(a)\neq 0\). Hence, \(u_a \neq 0\) is proved.
				
				\noindent Claim 3: $\mathcal{A}_F(u_n)\rightarrow \mathcal{A}_F(u_a)$ and $\bar{\eta}(u_n)\rightarrow \bar{\eta}(u_a)$.
				
				\noindent Since
				\(
				u_n\longrightarrow u_a
				~~\text{strongly in }
				L^{2+\frac{\alpha}{2}}(\mathbb R^2),
				\)
				the sequence 
				\(
				\left\{
				(I_\alpha*F(u_n))f(u_n)u_n
				\right\}
				\)
				is uniformly bounded in \(L^1(\mathbb R^2)\) and uniformly integrable, Vitali's convergence theorem implies that
				\[
				\bar{\eta}(u_n)=\int_{\mathbb R^2}
				(I_\alpha*F(u_n))f(u_n)u_n\,dx
				\longrightarrow
				\int_{\mathbb R^2}
				(I_\alpha*F(u_a))f(u_a)u_a\,dx=\bar{\eta}(u_a).
				\]
				Similarly, by \ref{f3},
				it follows that the sequence 
				\(
				\left\{
				(I_\alpha*F(u_n))F(u_n)
				\right\}
				\)
				is also uniformly integrable. Therefore, 
				\begin{equation}\label{Choquard-energy-convergence}
					\begin{aligned}
						\lim_{n\rightarrow+\infty}\mathcal{A}_F(u_n)&=\lim_{n\to+\infty}
						\int_{\mathbb R^2}
						(I_\alpha*F(u_n))F(u_n)\,dx
						=
						\int_{\mathbb R^2}
						\lim_{n\to+\infty}
						\bigl(I_\alpha*F(u_n)\bigr)F(u_n)\,dx
						\\
						&=
						\int_{\mathbb R^2}
						(I_\alpha*F(u_a))F(u_a)\,dx=\mathcal{A}_F(u_a).
					\end{aligned}
				\end{equation}
				Indeed, since \(u_n(x)\to u_a(x)\) almost everywhere in
				\(\mathbb R^2\), the continuity of \(F\), together with the 
				convergence properties of the Riesz potential established previously,
				gives
				\[
				\bigl(I_\alpha*F(u_n)\bigr)F(u_n)
				\longrightarrow
				\bigl(I_\alpha*F(u_a)\bigr)F(u_a)
				\text{ a.e. in }\mathbb R^2.
				\]
				Therefore, \eqref{Choquard-energy-convergence} follows from Vitali's
				convergence theorem.

				\noindent Claim 4: $u_a$ is a weak solution of 
				\begin{equation}\label{limit-equation-mixed}
					\mathcal{L} u_a+\lambda_a u_a
					=
					\Lambda(I_\alpha*F(u_a))f(u_a)
					\text{ in }\mathbb R^2.
				\end{equation}
				Indeed, since $J'(u_n)+\lambda_n u_n=o(1)$, for every
				\(\varphi\in C_c^\infty(\mathbb R^2)\), we have:
				\begin{equation*}
					\int_{\mathbb{R}^2}\nabla u_n \nabla \varphi dx +\ll u_n,\varphi\gg +\lambda_n \int_{\mathbb{R}^2}u_n \varphi dx = \Lambda \int_{\mathbb{R}^2}(I_{\alpha}*F(u_n))f(u_n)\varphi dx
				\end{equation*}
				Passing to the limit as \(n\to\infty\), using
				\[
				u_n\rightharpoonup u_a
				\text{ in }H^1(\mathbb R^2), \;\;
				\lambda_n\to\lambda_a,
				\]
				and the convergence of the nonlinear term established above, we obtain
				\begin{equation*}
					\int_{\mathbb{R}^2}\nabla u_a \nabla \varphi dx +\ll u_a,\varphi\gg +\lambda_a \int_{\mathbb{R}^2}u_a \varphi dx = \Lambda \int_{\mathbb{R}^2}(I_{\alpha}*F(u_a))f(u_a)\varphi dx
				\end{equation*}
				Thus, \(u_a\) is a weak solution of
				\eqref{limit-equation-mixed}.
				{
					We now show that $u_a$ satisfies the constraint as well.
				}
				
				{
					\noindent Claim 5: $\| u_a \|_2=a$.\\
					Testing \eqref{limit-equation-mixed} with \(u_a\), we obtain
					\begin{equation}\label{limit-testing-mixed}
						\|\nabla u_a\|_2^2
						+[u_a]_s^2
						+\lambda_a\|u_a\|_2^2
						=
						\Lambda
						\int_{\mathbb R^2}
						(I_\alpha*F(u_a))f(u_a)u_a\,dx.
					\end{equation}
					The Poho\v{z}aev identity gives
					\begin{equation}\label{limit-pohozaev-identity}
						(1-s)[u_a]_s^2
						+\lambda_a\|u_a\|_2^2
						=
						\Lambda\left(\frac{2+\alpha}{2}\right)
						\mathcal A_F(u_a).
					\end{equation}
					Hence
					\begin{equation}\label{limit-P-mixed}
						\|\nabla u_a\|_2^2+s[u_a]_s^2
						+
						\Lambda\left(\frac{2+\alpha}{2}\right)
						\mathcal A_F(u_a)
						-
						\Lambda
						\int_{\mathbb R^2}
						(I_\alpha*F(u_a))f(u_a)u_a\,dx
						=0.
					\end{equation}
					Since \(P(u_n)\to0\), Claim~3 and
					\eqref{limit-P-mixed} imply
					\[
					\|\nabla u_n\|_2^2+s[u_n]_s^2
					\longrightarrow
					\|\nabla u_a\|_2^2+s[u_a]_s^2.
					\]
					By weak lower semicontinuity,
					\[
					\|\nabla u_a\|_2^2
					\le
					\liminf_{n\to\infty}\|\nabla u_n\|_2^2
					~~~~~\text{and}~~~~
					[u_a]_s^2
					\le
					\liminf_{n\to\infty}[u_n]_s^2.
					\]
					Since \(s>0\), it follows that
					\(\|\nabla u_n\|_2^2
					\longrightarrow
					\|\nabla u_a\|_2^2
					\quad \text{and}~~~
					[u_n]_s^2
					\longrightarrow
					[u_a]_s^2.
					\)
					Testing
					\(J'(u_n)+\lambda_nu_n=o(1)
					\)
					with \(u_n\), and using \(\|u_n\|_2^2=a^2\), we obtain
					\[
					\|\nabla u_n\|_2^2+[u_n]_s^2+\lambda_na^2
					=
					\Lambda
					\int_{\mathbb R^2}
					(I_\alpha*F(u_n))f(u_n)u_n\,dx
					+o(1).
					\]
					Passing to the limit and using Claim~3 gives
					\[
					\|\nabla u_a\|_2^2+[u_a]_s^2+\lambda_aa^2
					=
					\Lambda
					\int_{\mathbb R^2}
					(I_\alpha*F(u_a))f(u_a)u_a\,dx.
					\]
					Comparing this identity with \eqref{limit-testing-mixed}, we obtain
					\[
					\lambda_a\bigl(a^2-\|u_a\|_2^2\bigr)=0.
					\]
					Since \(\lambda_a\neq0\), we get $\| u_a\|_2^2=a^2$  and hence \(u_a\in\mathcal S_r(a)\). Further, since
					\(
					u_n\rightharpoonup u_a
					\text{ in }L^2(\mathbb R^2)
					\)
					and
					\(
					\|u_n\|_2=\|u_a\|_2=a,
					\)
					we obtain
					\(
					u_n\to u_a \text{ strongly in }L^2(\mathbb R^2).
					\)
					Together with
					\(
					\|\nabla u_n\|_2^2
					\to
					\|\nabla u_a\|_2^2,
					\)
					this yields
					\(
					u_n\to u_a \text{ strongly in }H^1(\mathbb R^2).
					\)
				}
				This completes the proof.
			\end{proof}  

			\section{ Proofs of main Results}\label{proofs}
			
			\noindent In this section, we complete the proofs of the main results. We first
			develop the fibre-map analysis needed for the ground-state
			characterization under the additional assumption \ref{f6}, and relate
			the mountain-pass level \(\gamma(a)\) to the least energy level \(m(a)\)
			on the Pohozaev manifold. We then prove Theorem~\ref{T1} by using the
			compactness results established in Section~\ref{PSa}, and finally deduce
			Theorem~\ref{T2}.
			
			\begin{lemma}\label{uni-mixed}
				Assume that \ref{f0}-\ref{f3} and \ref{f6} hold, and let $u\in \mathcal{S}_{r}(a)$. Then the functional
				\[
				\widetilde{J}_{u}(\sigma):=J(H(u,\sigma)), \qquad \sigma\in \mathbb{R},
				\]
				admits a unique maximum point $\sigma(u)\in \mathbb{R}$ such that
				\(
				H(u,\sigma(u))\in \mathcal{P}(a).
				\)
				Moreover, for $v\in \mathcal{P}(a)$, $\sigma(v)=0.$
			\end{lemma}
			\begin{proof}
				Let $u\in S_r(a)$ and $\sigma\in \mathbb{R}$, then we have:
				\begin{eqnarray*}
					\tilde{J}(\sigma) & = & J(H(u,\sigma)) 
					=  \frac{\|\nabla H(u,\sigma)\|_2^2}{2}+\frac{[H(u,\sigma)]_s^2}{2}-\frac{\Lambda}{2}\mathcal{A}_F(H(u,\sigma))\\
					& = & \frac{e^{2\sigma}}{2}\|\nabla u \|_2^2 +\frac{e^{2s\sigma}}{2}[u]_s^2 -\frac{\Lambda}{2e^{(2+\alpha)\sigma}}\mathcal{A}_F(e^{\sigma}u).
				\end{eqnarray*}
				thus,
				\begin{eqnarray*}
					\tilde{J}'_u(\sigma) & = & e^{2\sigma}\|\nabla u\|_2^2+se^{2s\sigma}[u]_s^2-\frac{\Lambda}{e^{(2+\alpha)\sigma}}\int_{\mathbb{R}^2}(I_{\alpha}*F(e^{\sigma}u)\tilde{F}(e^{\sigma}u)\\
					& = &e^{2\sigma}\left(\|\nabla u\|_2^2 +se^{2(s-1)\sigma}[u]_s^2-\frac{\Lambda}{e^{(4+\alpha)\sigma}}\int_{\mathbb{R}^2}(I_{\alpha}*F(e^{\sigma}u)\tilde{F}(e^{\sigma}u)\right)
					=  e^{2\sigma}\phi(\sigma),
				\end{eqnarray*}
				where 
				\begin{eqnarray*}
					\phi(\sigma) & := & \|\nabla u\|_2^2 +se^{2(s-1)\sigma}[u]_s^2-\frac{\Lambda}{e^{(4+\alpha)\sigma}}\int_{\mathbb{R}^2}(I_{\alpha}*F(e^{\sigma}u)\tilde{F}(e^{\sigma}u)\\
					& = & \|\nabla u\|_2^2 +se^{2(s-1)\sigma}[u]_s^2-\Lambda \int_{\mathbb{R}^2}\left(I_{\alpha}*\frac{F(e^{\sigma }u)}{e^{(2+\frac{\alpha}{2})\sigma}}\right)\frac{\tilde{F}(e^{\sigma}u)}{e^{(2+\frac{\alpha}{2})\sigma}}dx.
				\end{eqnarray*}
				On the one hand, \ref{f6} implies that 
				$$\zeta\mapsto \frac{\tilde{F}(\zeta t)}{\zeta^{2+\frac{\alpha}{2}}}\text{ is non decreasing for any } t\in \R.$$
				On the other hand,
				\[\frac{d}{d \zeta}\left(\frac{F(\zeta t)}{\zeta^{2+\frac{\alpha}{2}}}\right) =\frac{f(\zeta t)\zeta t -(2+\frac{\alpha}{2})F(\zeta t)}{\zeta^{3+ \frac{\alpha}{2}}}>0.
				\]
				Hence, the function  
				$$\sigma\mapsto -\Lambda \int_{\mathbb{R}^2}\left(I_{\alpha}*\frac{F(e^{\sigma }u)}{e^{(2+\frac{\alpha}{2})\sigma}}\right)\frac{\tilde{F}(e^{\sigma}u)}{e^{(2+\frac{\alpha}{2})\sigma}}dx \text{ is strictly increasing in }\mathbb{R}.$$
				Thus, it can vanish at most once. Therefore, $\tilde{J}_u$ can have at most one critical point. 
				
				\noindent Moreover, by continuity of the exponential function, we can always find $\sigma_0$ such that
				$$T(H(u,\sigma))=\|\nabla H(u,\sigma_0)\|_2^2+[H(u,\sigma_0)]_s^2 = e^{2\sigma_0}\|\nabla u\|_2^2 +e^{2s\sigma_0}[u]_s^2 = 2K(a)
				,$$
				thus $H(u,\sigma_0)\in B$. Then by Lemma \ref{geo_mixed} and \ref{P1_mixed}, we get
				$$\tilde{J}_u(\sigma)\rightarrow 0 \text{ as } \sigma \rightarrow -\infty;\; \tilde{J}_u(\sigma)\rightarrow -\infty \text{ as } \sigma \rightarrow+\infty \text{ and } \tilde{J}_{u}(\sigma_0)>0.$$
				Thus, $\tilde{J}_u$ must have at least one critical point. \\
				\noindent Therefore, $\tilde{J}_u$ has the following curvature with the  unique critical point $\sigma(u)$.
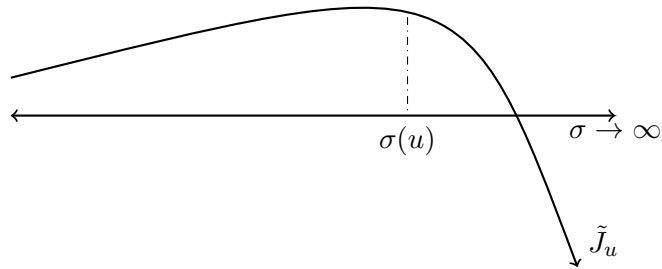
				\begin{center}
					\begin{tikzpicture}[scale=1.1]
						
						\draw[very thick,->] (-4,0) -- (4.3,0)
						node[right] {$\sigma$};
						
						\draw[
						very thick,
						blue,
						->,
						smooth
						]
						(-3.8,0.35)
						.. controls (-1.5,1.2) and (0.5,2.2)
						.. (1.25,1.55)
						.. controls (2.2,1.1) and (3.0,-0.8)
						.. (3.8,-2.0)
						node[below right,black] {$\widetilde{J}_u(\sigma)$};
						
						\fill[red] (1.25,1.55) circle (2.5pt);
						
						\draw[
						red,
						dashed,
						thick
						]
						(1.25,1.55)--(1.25,0);
						
						\node[red,below] at (1.25,0)
						{$\sigma(u)$};
						
						\node[below] at (4,0)
						{$\sigma\to+\infty$};
						
					\end{tikzpicture}
				\end{center}
				Clearly, $\sigma(u)$ corresponds to the global maximum of $\tilde{J}_u$. Further, since
				\begin{equation*}
					P(H(u,\sigma(u)) =  \tilde{J}'_{u}(\sigma(u))=0,
				\end{equation*}
				we get $H(u,\sigma(u))\in \mathcal{P}(a)$. Moreover, if $v\in \mathcal{P}(a)$, then
				\begin{equation*}
					\tilde{J}'_v(0) =  P(H(v,0))= P(v)=0,
				\end{equation*}
				thus $\sigma(v)=0.$
		\end{proof}}     
		\begin{lemma}\label{fin-mixed}
			Assume that \ref{f1}-\ref{f2} and \ref{f4} hold. Then
			\(
			\gamma(a)=m(a),
			\)
			where
			\[
			m(a):=\inf_{u\in\mathcal P(a)}J(u).
			\]
			Moreover,
			\[
			m(a)=
			\inf_{u\in\mathcal P(a)}
			\max_{\sigma\in\mathbb R}J(H(u,\sigma)).
			\]
		\end{lemma}
		\begin{proof}
			We first observe that, by Lemma \ref{uni-mixed}, if \(v\in\mathcal P(a)\), then
			the map
			\[
			\sigma\mapsto J(H(v,\sigma))
			\]
			attains its unique maximum at \(\sigma=0\). Indeed, \(H(v,0)=v\in\mathcal P(a)\).
			Therefore,
			\[
			J(v)=\max_{\sigma\in\mathbb R}J(H(v,\sigma)).
			\]
			Taking the infimum over \(v\in\mathcal P(a)\), we obtain
			\[
			m(a)
			=
			\inf_{v\in\mathcal P(a)}J(v)
			=
			\inf_{v\in\mathcal P(a)}
			\max_{\sigma\in\mathbb R}J(H(v,\sigma)).
			\]
			We now prove that \(\gamma(a)\le m(a)\). Let \(v\in\mathcal P(a)\) be arbitrary.
			By Lemma \ref{geo_mixed}, we have
			\[
			J(H(v,\sigma))\to 0
			\text{ as }\sigma\to-\infty
			\]
			and
			\[
			J(H(v,\sigma))\to-\infty
			\text{ as }\sigma\to+\infty.
			\]
			Moreover, the scaling \(H(v,\sigma)\) preserves the \(L^2\)-norm, and hence
			\[
			H(v,\sigma)\in\mathcal S_r(a)
			\text{ for every }\sigma\in\mathbb R.
			\]
			Thus, we may choose \(\sigma_1<0<\sigma_2\) such that
			\[
			H(v,\sigma_1)\in A
			\text{ and }
			J(H(v,\sigma_2))<0.
			\]
			Define
			\[
			h(t):=H(v,(1-t)\sigma_1+t\sigma_2),
			\qquad t\in[0,1].
			\]
			Then \(h\in\Gamma\). Hence, by the definition of \(\gamma(a)\),
			\[
			\gamma(a)
			\le
			\max_{t\in[0,1]}J(h(t)).
			\]
			Since \(h(t)\) is only a part of the fiber curve \(\sigma\mapsto H(v,\sigma)\), we have
			\[
			\max_{t\in[0,1]}J(h(t))
			\le
			\max_{\sigma\in\mathbb R}J(H(v,\sigma)).
			\]
			Using again that \(v\in\mathcal P(a)\), Lemma \ref{uni-mixed} gives
			\[
			\max_{\sigma\in\mathbb R}J(H(v,\sigma))=J(v).
			\]
			Therefore,
			\[
			\gamma(a)\le J(v).
			\]
			Taking the infimum over \(v\in\mathcal P(a)\), we obtain
			\[
			\gamma(a)\le \inf_{v\in\mathcal P(a)}J(v)=m(a).
			\]
			It remains to prove that \(\gamma(a)\ge m(a)\). Let \(h\in\Gamma\) be arbitrary.
			By the definition of \(\Gamma\), we have
			\[
			h(0)\in A
			\text{ and }
			J(h(1))<0.
			\]
			By choosing \(K(a)>0\) sufficiently small in the definition of \(A\), the same
			estimates used in Lemma \ref{geo_mixed} imply that
			\[
			P(w)>0
			\text{ for every } w\in A.
			\]
			Hence,
			\[
			P(h(0))>0.
			\]
			On the other hand, we claim that
			\[
			J(w)<0 \quad \Longrightarrow \quad P(w)<0.
			\]
			Indeed, let \(w\in\mathcal S_r(a)\) and consider
			\[
			\varphi_w(\sigma):=J(H(w,\sigma)).
			\]
			By Lemma \ref{uni-mixed}, \(\varphi_w\) has a unique maximum point
			\(\sigma(w)\), and
			\[
			\varphi_w'(\sigma)=P(H(w,\sigma)).
			\]
			Moreover, \(\varphi_w(\sigma)\to0\) as \(\sigma\to-\infty\). If \(P(w)\ge0\),
			then \(0\) lies before or at the maximum point of \(\varphi_w\), and consequently
			\[
			J(w)=\varphi_w(0)\ge0,
			\]
			which contradicts \(J(w)<0\). Therefore \(P(w)<0\). Applying this to \(w=h(1)\),
			we get
			\(
			P(h(1))<0.
			\)
			Since \(h\) is continuous in \(\mathcal S_r(a)\) and \(P\) is continuous on
			\(\mathcal S_r(a)\), the map
			\[
			t\mapsto P(h(t))
			\]
			is continuous on \([0,1]\). Since
			\[
			P(h(0))>0
			\text{ and }
			P(h(1))<0,
			\]
			the intermediate value theorem gives \(t_0\in(0,1)\) such that
			\(
			P(h(t_0))=0.
			\)
			Since \(h(t_0)\in\mathcal S_r(a)\), we have
			\[
			h(t_0)\in\mathcal P(a).
			\]
			Therefore,
			\[
			\max_{t\in[0,1]}J(h(t))
			\ge
			J(h(t_0))
			\ge
			\inf_{v\in\mathcal P(a)}J(v)
			=
			m(a).
			\]
			Taking the infimum over \(h\in\Gamma\), we obtain
			\[
			\gamma(a)\ge m(a).
			\]
			Combining
			\[
			\gamma(a)\le m(a)
			\text{ and }
			\gamma(a)\ge m(a),
			\]
			we conclude that
			\[
			\gamma(a)=m(a).
			\]
		\end{proof}
		\begin{myproof}{Theorem}{\ref{T1}}
			As we see in Section \ref{PSa}, there exists a $(PSP)$-sequence for $J$
			constrained to $\mathcal{S}_{r}(a)$. It follows from Lemma \ref{P1_mixed} that
			$\gamma(a)>0$. Furthermore, by Corollary \ref{upperbound_mixed}, we have
			\[
			\gamma(a)< \frac{(2+\alpha)\pi}{2\gamma_0}.\]
			{
				Then the compactness Lemma~\ref{uniform-bound} applies and there exist
				\(u_a\in H^1_{\rm rad}(\mathbb R^2)\setminus\{0\},
				~~\lambda_a\in\mathbb R,
				\)
				such that
				\[
				\mathcal Lu_a+\lambda_a u_a
				=
				\Lambda(I_\alpha*F(u_a))f(u_a)
				\qquad\text{in }\mathbb R^2.
				\]
				Moreover,
				\(
				\lambda_a\bigl(a^2-\|u_a\|_2^2\bigr)=0.
				\)
				If \(\lambda_a\neq0\), then
				\(
				\|u_a\|_2^2=a^2,
				\)
				and therefore
				\(
				u_a\in\mathcal S_r(a).
				\)
				Consequently,
				\(
				(u_a,\lambda_a)
				\in
				\mathcal S_r(a)\times(\mathbb R\setminus\{0\})
				\)
				is a normalized solution of \eqref{aa}.
				On the otherhand, if \(\lambda_a=0\), the asserted upper bound is immediate.
			}
			
			\noindent We are left to estimate the Lagrange multiplier $\lambda_a$.
			Since $u_a$ is a constrained critical point of $J$ on
			$\mathcal{S}_r(a)$, it satisfies
			\(
			J'|_{\mathcal S_r(a)}(u_a)=0,
			\)
			that is,
			\[
			J'(u_a)+\lambda_a u_a=0 \text{ in }H^{-1}(\mathbb R^2).
			\]
			Testing the equation \eqref{aa} with $u_a$, and using
			$\|u_a\|_2^2=a^2$, we obtain
			\begin{equation}\label{lambda1}
				\|\nabla u_a\|_2^2
				+[u_a]_s^2
				+\lambda_a a^2
				=
				\Lambda
				\int_{\mathbb R^2}
				(I_\alpha*F(u_a))
				f(u_a)u_a\,dx .
			\end{equation}
			On the other hand, since $u_a\in\mathcal P(a)$, the Pohozaev identity
			gives
			\begin{equation}\label{pohozaev-lambda}
				\|\nabla u_a\|_2^2
				+s[u_a]_s^2
				+\Lambda\left(\frac{2+\alpha}{2}\right)
				\mathcal A_F(u_a)
				=
				\Lambda
				\int_{\mathbb R^2}
				(I_\alpha*F(u_a))
				f(u_a)u_a\,dx,
			\end{equation}
			Substituting \eqref{pohozaev-lambda} into
			\eqref{lambda1}, we obtain
			\begin{equation}\label{lambda2}
				\lambda_a a^2
				=
				(s-1)[u_a]_s^2
				+
				\Lambda
				\left(\frac{2+\alpha}{2}\right)
				\mathcal A_F(u_a)\leq \Lambda\left(\frac{2+\alpha}{2} \right)\mathcal{A}_F(u_a),
			\end{equation}
			since \(0<s<1\). Now, using the identity
			\(
			J(u_a)=\gamma(a),
			\)
			we have
			\[
			\Lambda
			\mathcal A_F(u_a)
			=
			\|\nabla u_a\|_2^2
			+[u_a]_s^2
			-
			2\gamma(a).
			\]
			Substituting this into \eqref{lambda2}; and using \eqref{eqgamma} and Corollary \ref{upperbound_mixed}, we obtain
			\begin{eqnarray*}
				\lambda_a & \leq & \left(\frac{2+\alpha}{2a^2}\right)\left(T(u_a)-2\gamma(a)\right)= \left(\frac{2+\alpha}{2a^2}\right) \left(\lim_{n\rightarrow+\infty}T(u_n)-2\gamma(a)\right)\\
				& \leq & \left(\frac{2+\alpha}{2a^2}\right) \left(\left(\frac{2\left(2\mu-\left(2+\alpha\right)\right)}
				{2\mu-(4+\alpha)}\right)\gamma(a)-2\gamma(a)\right)\\
				& = & \frac{2(2+\alpha)}{a^2(2\mu-4-\alpha)}\gamma(a)< \frac{2(2+\alpha)}{a^2(2\mu-4-\alpha)}\beta \\
				& = & \frac{(2+\alpha)^2\pi}{\gamma_0a^2(2\mu-4-\alpha)}
			\end{eqnarray*}
			which is desired bound.
		\end{myproof}

		
		\noindent Finally we prove Theorem \ref{T2}.
		\begin{myproof}{Theorem}{\ref{T2}}	
			Assume, in addition, that \ref{f6} holds. Let
			\(
			(u_a,\lambda_a)\in\mathcal S_r(a)\times\mathbb R
			\)
			be the normalized solution obtained in
			Theorem~\ref{T1} {
			}. In particular,
			\(
			J(u_a)=\gamma(a)
			\)
			and, since \(u_a\) is a normalized solution, the Poho\v{z}aev
			identity gives
			\(
			P(u_a)=0.
			\)
			Therefore,
			\(
			u_a\in\mathcal P(a).
			\)
			From Lemma \ref{fin-mixed}, 
			\(
			\gamma(a)=m(a),
			\)
			where
			\(
			m(a):=\inf_{u\in\mathcal P(a)}J(u).
			\)
			Thus, \(u_a\) has the least energy among all normalized solutions of
			problem \ref{aa}. Therefore, \(u_a\) is a normalized ground
			state solution. 
		\end{myproof}

		\section*{Declarations}
		\noindent\textbf{Ethical Approval.}
		Not applicable.
		
		\medskip
		\noindent\textbf{Competing interests.}
		The authors declare that they have no competing interests.
		
		\medskip
		\noindent\textbf{Authors contributions.}
		The authors contributed equally to this work.
		
		\medskip
		\noindent\textbf{Availability of data and materials.}
		Data sharing is not applicable to this article as no new data were created or analyzed in this study.

		\section*{Acknowledgements}
		\noindent The author, Nidhi Nidhi (PMRF ID-142685), acknowledges the support provided by the Ministry of Education, Government of India, under the Prime Minister's Research Fellows scheme.
		
		

	\end{document}